\documentclass[final]{siamltex} 
\usepackage{epsfig}
\usepackage{amsmath,amssymb,bm}
\usepackage{epic}
\usepackage{comment}
\def\Hneg#1{\widetilde{H}^{-#1}}
\newcommand{\Abs}[1]{\left|#1\right|}
\newcommand{\abs}[1]{\lvert#1\rvert}

\newcommand{\binnprd}[2]{\left( #1 , #2 \right)}
\newcommand{\innprd}[2]{\left< #1 , #2 \right>}

\def\Re{{\mathrm{Re}}}
\def\be{\begin{eqnarray}}
\def\ee{\end{eqnarray}}
\def\bes{\begin{eqnarray*}}
\def\ees{\end{eqnarray*}}
\def\R{\mathbb{R}}
\def\C{\mathbb{C}}
\def\H{\mathbb{H}}

\def\norm#1{\left|\!\left| #1 \right|\!\right|}
\def\nnorm#1{|\!| #1 |\!|}

\def\op#1{{\mathcal #1}}
\def\mat#1{{\bf #1}}
\def\vect#1{{\bf #1}}

\def\enorm#1{|\!|\!| #1 |\!|\!|}

\newtheorem{condition}[theorem]{Condition}

\newtheorem{remark}[theorem]{Remark}
\newtheorem{algorithm}[theorem]{Algorithm}

\title{Multigrid preconditioning of linear systems for semismooth Newton methods applied to
optimization problems constrained by smoothing operators}

\author{Andrei Dr{\u{a}}g{\u{a}}nescu\thanks{Department
   of Mathematics and Statistics, University of Maryland, Baltimore
   County, 1000~Hilltop Circle, Baltimore, Maryland 21250 ({\tt draga@umbc.edu}).}
This work was
supported  by the Department of Energy under contract
no. DE-SC0005455,  and by the National Science Foundation under award
DMS-1016177.  Computations were performed in part on the UMBC's High
Performance Computing Facility which was supported in part by the
National Science Foundation under awards CNS-0821258 and DMS-0821311.}

\begin{document}

\maketitle

\begin{abstract}
This article is concerned with the question of constructing efficient multigrid
preconditioners for the linear systems arising when applying semismooth
Newton methods to large-scale  linear-quadratic optimization
problems  constrained by smoothing operators with box-constraints on
the controls. It is shown that, for certain  discretizations of the
optimization problem, the linear systems to be solved at each
semismooth Newton iteration reduce to inverting principal minors of the
Hessian of the associated unconstrained problem. As in the case when
box-constraints on the controls are absent, the multigrid
preconditioner introduced here is shown to increase in quality as the
mesh-size decreases, resulting in a number of iterations that
decreases with mesh-size. However, unlike the unconstrained case, the
spectral distance between the preconditioners and the Hessian is shown to 
be of suboptimal order in general.
\end{abstract}



\pagestyle{myheadings}
\thispagestyle{plain}
\markboth{ANDREI DR{\u{A}}G{\u{A}}NESCU}{MULTIGRID PRECONDITIONERS 
FOR SEMISMOOTH NEWTON METHODS}

\section{Introduction}
\label{sec:intro}
The objective of this article is to develop  efficient multigrid preconditioners
for the linear systems arising in the  solution process 
of large-scale optimal control problems  constrained by partial differential equations (PDEs) using semismooth Newton methods (SSNMs). The model problems 
under scrutiny have the form
\begin{equation}
\label{eq:reducedoptprob}
\left\{\begin{array}{lll}
\vspace{4pt}
\mathrm{minimize} &\op{J}_{\beta}(u)\stackrel{\mathrm{def}}{=}\frac{1}{2}\nnorm{\op{K}u-y_d}^2+ \frac{\beta}{2}\nnorm{u}^2,&\ \ \beta>0\ \mathrm{fixed}\\
\vspace{4pt}
\mathrm{subject\  to:\ }& a \le u\le b\ \ \mathrm{a.e.},
\end{array}\right .
\end{equation}
where $\Omega\subset \R^d$ is a bounded, open set, $\op{K}:\op{U}\to\op{U}$ is a linear compact operator in the Hilbert space $\op{U}=L^2(\Omega)$,
and $a, b, y_d\in\op{U}$ are given functions so that $a(x)<b(x)$ for almost all $x\in \Omega$. 
More detailed conditions will be given in Section~\ref{sec:mainresults}. 

In the PDE-constrained optimization literature~$\op{K}$  appears oftentimes 
as the composition of two operators $\op{K}=I_{\op{Y}\to \op{U}} \op{S}$, where $I_{\op{Y}\to \op{U}}:\op{Y}\to\op{U}$ is the compact embedding
of a space~$\op{Y}$ into~$\op{U}$, and  $\op{S}:\op{U}\to\op{Y}$ is the solution operator of a 
linear PDE: if $e:\op{Y}\times \op{U}\rightarrow \op{Y}^*$ defines the linear PDE 
\bes
\label{eq:genpde}
e(y,u) = 0\ , 
\ees
then $e(y,u) = 0$ if and only if $y = \op{S} u$. This way we obtain an equivalent formulation of~\eqref{eq:reducedoptprob} 
that has become standard in the PDE-constrained optimization
literature~\cite{MR2516528}:
\begin{equation}
\label{eq:genpdecopt}
\left\{\begin{array}{l}
\vspace{6pt}
\mathrm{minimize} \ \ \frac{1}{2}\nnorm{y-y_d}^2+ \frac{\beta}{2}\nnorm{u}^2\ \\
\mathrm{subject\ to:\ } e(y,u)=0,\  u\in \op{U}_{\mathrm{ad}}=\{u\in \op{U}\ :\ a\le u\le b\ \  \mathrm{a.e.}\}\ .\end{array}\right .
\end{equation}
In this formulation $u$ is the control and we shall call $y$ the state.  We shall also refer to~\eqref{eq:reducedoptprob} as the reduced form
of~\eqref{eq:genpdecopt}. There are many applications, however, for which the formulation~\eqref{eq:reducedoptprob} may not correspond 
to a problem of the form~\eqref{eq:genpdecopt}, as is the case when $\op{K}$ is an explicitly integral operator as in image deblurring. In other instances
it may simply be that the reduced form~\eqref{eq:reducedoptprob} is
more natural than~\eqref{eq:genpdecopt}. For example, in the solution process of the backwards parabolic equation discussed in~\cite{MR2429872} 
(see also the related work~\cite{Bart_Omar:SC05a}) $\op{K}u$ represents the time-$T$ solution of a linear parabolic equation with initial value~$u$; 
hence, the purpose of the optimization  problem is to find an initial value $u$ for which the time-$T$ solution $\op{K}u$ is close to a given
state $y_d$. While this problem can be formulated as a PDE-constrained optimization~\eqref{eq:genpdecopt} with space-time constraints, in a truly large-scale, 
four-dimensional setting the three-dimensional reduced form~\eqref{eq:reducedoptprob} of the optimization problem is more trackable. We should remark that 
both for the image-deblurring and the backwards heat-equation examples the operator $\op{K}$ can be discretized at various resolutions, and the 
adjoint $\op{K}^*$ can be computed at a cost comparable to that of $\op{K}$. Hence, in this work we focus on the reduced problem~\eqref{eq:reducedoptprob}.

Due to the availability of increasingly powerful parallel computers, the scientific community  has shown 
a growing interest over the last decade  in  developing scalable solvers  for 
large-scale optimization problems with PDE constraints. Multigrid  methods have long been associated  with large-scale linear systems, the 
paradigm being that the solution process can be significantly accelerated by using  multiple resolutions of the same problem.
However, the exact embodiment of the  multigrid paradigm depends
strongly on the class of problems considered, with multigrid methods for differential
equations (elliptic, parabolic, flow problems) being significantly different from
methods for integral equations.
To place our problem in the context of multigrid methods we consider the simplified version of~\eqref{eq:reducedoptprob}
obtained by removing the inequality constraints on the
control, e.g. $\op{U}_{\mathrm{ad}}=\op{U}$, case in which~\eqref{eq:reducedoptprob}
reduces to the linear system 
\be
\label{eq:normaleq}
(\op{K}^*\op{K}+\beta I) u = \op{K}^* y_d\ ,
\ee
which represents the normal equations associated with the Tikhonov regularization of the ill-posed problem
\be
\label{eq:illposedorig}
\op{K} u = y_d\ .
\ee
Beginning with the works of Hackbusch~\cite{hackbusch81} (see also~\cite{MR1350296})
much effort  has been devoted to developing efficient multigrid methods for
solving equations like~\eqref{eq:normaleq} and~\eqref{eq:illposedorig}, e.g. see~\cite{MR1151773, MR97k:65299,
MR2001h:65069, MR1986801, MR2421947, MR2429872} and the references
therein. For example, Dr{\u a}g{\u a}nescu and Dupont~\cite{MR2429872}  have constructed a multigrid preconditioner $\op{M}_h$ 
 which satisfies
\be
\label{eq:mgoptimality}
1-C \frac{h^p}{\beta}\le \frac{\innprd{\op{G}_h u}{u}}{\innprd{\op{M}_hu}{u}}\le 1+C \frac{h^p}{\beta},\ \ \ \forall u\ne 0,
\ee
where $h$ is the mesh-size, $\op{G}_h$ is the discretized version of $(\op{K}^*\op{K}+\beta I)$, and
$p$ is the order of the discretization ($p=2$ for piecewise linear finite elements).
We regard~\eqref{eq:mgoptimality} as optimal-order scalability since it implies that the condition number
of the unpreconditioned system  $\mathrm{cond}(\op{G}_h)=O(1/\beta)$ is being reduced by a factor of $h^p$ 
in the $\op{M}_h$-preconditioned system, namely $\mathrm{cond}(\op{M}_h^{-1}\op{G}_h)=O(h^p/\beta)$; 
this reduction is of optimal order given that the discretization order is $p$.
A similar situation is encountered in classical multigrid for elliptic problems 
where multigrid is used to reduce the condition number from $O(h^{-2})$ to $O(1)$, the latter implying
the desired mesh-independence property. We should point out that~\eqref{eq:mgoptimality} implies that
the number of iterations actually \emph{decreases} with $h\downarrow 0$ to the point where, asymptotically, only one 
iteration is needed on very fine meshes.

The presence of explicit box constraints on the controls and/or states in PDE-constrained
optimization problems is sometimes critical both for practical (design
constraints) and theoretical reasons (e.g.  states representing
densities or concentrations of substances that have to be nonnegative).  Methods for solving optimization
problems with inequality constraints are fundamentally different and
more involved than those for unconstrained problems; they  generally
fall into  two competing categories: \emph{interior point methods} (IPMs) and
active-set methods such as SSNMs.
Both types of methods exhibit superlinear local convergence
and can be formulated and analyzed  both  in finite dimensional spaces as well as in function
spaces~\cite{MR1972217, MR2421314, MR2193506}, the latter being a critical step towards
proving mesh-independence for the number of optimization steps. Both IPMs and SSNMs are iterative procedures
that require  the equivalent of a few PDE solves, i.e., applications of $\op{K}$,
for each iteration (called here {\em outer iteration}) and the solution one or two {\em inner} linear systems; for \mbox{SSNMs} only one inner
linear solve is required,  
while for Mehrotra's predictor-corrector IPM implementation two inner linear solves are needed for each outer iteration.
Efficiency of the
solution process is measured  by the number of outer iterations (ideally mesh-independent) needed
to solve the problem to a desired tolerance  and by the ability to
solve the inner linear systems efficiently. In this work we concentrate on the latter.

Even though SSNMs and IPMs essentially solve the same problem, the linear algebra requirements
for IPMs are different from those of SSNMs. If formulated in the reduced form~\eqref{eq:reducedoptprob}, 
that is, with the PDE constraints eliminated (e.g. when the application of~$\op{K}$ is treated as a black-box) 
the structure of the systems arising in the IPM solution process is shown to be similar to
the system~\eqref{eq:normaleq} for the unconstrained problem~\cite{Dra:Pet:ipm}. More precisely, for IPMs
we need to solve systems of the form
\be
\label{eq:linsysipm}
\op{G}_{h,\lambda}u\stackrel{\mathrm{def}}{=}(\op{K}^*\op{K} +\op{D}_{\lambda}) u = b\ ,
\ee
where $\op{D}_\lambda$ is the multiplication operator with a relatively smooth function $\lambda$.
Moreover, under specific conditions and for a  natural
discrete formulation of~\eqref{eq:reducedoptprob}, Dr{\u a}g{\u a}nescu and Petra~\cite{Dra:Pet:ipm} have  
constructed  multigrid preconditioners for the linear systems~\eqref{eq:linsysipm} that exhibit
a certain degree of optimality similar to the one multigrid 
preconditioners in~\cite{MR2429872}: the resulting multigrid preconditioner $\op{M}_h$ for the $\op{G}_{h,\lambda}$ satisfies
\be
\label{eq:mgoptimalityipm}
1-C \frac{h^2}{\beta}\nnorm{\lambda^{-\frac{1}{2}}}_{W^2_{\infty}}\le 
\frac{\innprd{\op{G}_{h,\lambda} u}{u}}{\innprd{\op{M}_hu}{u}}\le 1+C \frac{h^2}{\beta}\nnorm{\lambda^{-\frac{1}{2}}}_{W^2_{\infty}},
\ \ \ \forall u\ne 0\ .
\ee
We recognize in~\eqref{eq:mgoptimalityipm} the optimal-order $h^2$-term (linear splines were used for discretization), 
but also remark that the quality of the preconditioner normally is affected by the lack of smoothness of $\lambda$: in general,
as the solution approaches the boundary, the smoothness of $\lambda$ is expected to degrade.

In this article we use similar ideas to design preconditioners for the linear systems arising in the SSNM solution process. 
Our main contribution is two-fold: we  construct  suitable coarse spaces and transition operators, and we give a detailed analysis of
the resulting two- and multigrid preconditioners. We note that the basic elements of our analysis are related to 
those in~\cite{Dra:Pet:ipm}, so in this sense the present work could be regarded as a 
companion of~\cite{Dra:Pet:ipm}. However, we should point out that the algorithms for SSNMs are essentially different from those
for IPMs, and so are the structures of their analyses.  For SSNMs we will show that the linear systems to be solved are
essentially principal subsystems of~\eqref{eq:normaleq} where the selected rows (and columns) correspond to the constraints that are 
deemed inactive at some point in the solution process. The constructed multigrid preconditioner is shown to
essentially satisfy~\eqref{eq:mgoptimality} with $p=\frac{1}{2}$. While this order of approximation is clearly suboptimal, it 
still brings a significant reduction of the condition number if $\sqrt{h}\ll \beta$, and still results in a solution process that requires
fewer and fewer inner linear iterations as $h\downarrow 0$.

The method developed and analyzed in this article is related to the multigrid method of the second kind developed by Hackbusch 
(see~\cite{MR814495}, Ch. 16). In fact, when using the multigrid iteration described in~\cite{MR814495} in connection with 
the coarse spaces and transition operators described here we observe that the number of inner linear iterations decreases as~$h\downarrow 0$. However,
our numerical experiments in Section~\ref{sec:numerics} indicate that our method is more efficient in absolute terms.
We should also note that the coarse spaces defined in this work are related to the symmetric multigrid
preconditioner developed by Hoppe and Kornhuber in~\cite{MR1276702} for obstacle problems, where the matrices to be preconditioned 
are subsystems of elliptic operators.

While our strategy is mainly designed for the reduced form~\eqref{eq:reducedoptprob}, a significant literature is devoted to 
multigrid methods applied to the complementarity problem representing the Karush-Kuhn-Tucker (KKT) system
of~\eqref{eq:genpdecopt}. Of these techniques we mention the collective smoothing multigrid method of Borz{\`{\i}} and Kunisch~\cite{MR2160699}.
For further references we refer the reader to the review article of Borz{\`{\i}} and Schulz~\cite{MR2505585}. Also, in a recent
article Stoll and Wathen~\cite{stollwathen} develop preconditioners  for the linear systems arising in the SSNMs solution process
of the unreduced optimal control problem~\eqref{eq:genpdecopt}; in their approach the linear systems are indefinite (since they
correspond to derivatives of Lagrangians) and sparse, since it is not $\op{K}$, but $\op{K}^{-1}$
that is explicitly present in the system. Yet another alternative strategy for solving the linear systems for SSNMs applied to the unreduced
problem~\eqref{eq:genpdecopt}, presented by Ulbrich in~\cite{MR1972217}, p.~219ff (see also~\cite{Ulbrichhabilitation})
involves reducing the linear systems to solving the discrete PDEs for which  efficient solvers are assumed to be readily available
(such as classical multigrid for elliptic problems). The question of which method is the most efficient is difficult to answer
for a general setting. However, we emphasize that the technique proposed in this article will work when the operator $\op{K}$ is given only as a black-box, or when
solvers are available for computing $\op{K}u$ efficiently.

This article is organized as follows:
in Section~\ref{sec:mainresults} we give the formal introduction of the problem, briefly discuss SSNMs, 
and present the main results.  Section~\ref{sec:analysis} is essentially devoted to proving the main result,
Theorem~\ref{maintheorem}, concerning the two-grid preconditioner, 
while in Section~\ref{sec:mganalysis} we extend the analysis to the multigrid preconditioner. 
In Section~\ref{sec:numerics} we show some numerical results to support 
our theoretical work, and we formulate some conclusions in Section~\ref{sec:conclusions}. Appendix~\ref{sec:specdist}
contains some technical results used in Section~\ref{sec:mganalysis}.

\section{Problem formulation and main results}
\label{sec:mainresults}
Our solution strategy follows the discretize-then-optimize paradigm, where
we first formulate a discrete optimization problem associated with~\eqref{eq:reducedoptprob},
which we then solve using SSNMs. After introducing the discrete
framework in Section~\ref{ssec:notation}, we discuss the optimality conditions and their semismooth formulation
in Section~\ref{ssec:ssnm}. In Section~\ref{ssec:primaldual} we derive the
linear systems needed to be solved at each  SSNM iteration.
The two-grid preconditioner and main two-grid results are given in Section~\ref{ssec:multigriddef}.
Furthermore, we  discuss the multigrid preconditioner in Section~\ref{ssec:multigrid}.

\subsection{Notation and discrete problem formulation}
\label{ssec:notation}
Let $\Omega\subset \R^d$ ($d=1, 2,$ or $3$) be a bounded domain which, for simplicity, we assume to be polygonal (if $d=2$) or
polyhedral (for $d=3$).
We denote by  $W_p^m(\Omega), H^m(\Omega),
H_0^m(\Omega)$ (with $p\in[1,\infty], m\in \mathbb{N}$) the standard
Sobolev spaces, and by \mbox{$\nnorm{\cdot}$} and
$\innprd{\cdot}{\cdot}$ the $L^2$-norm and inner product,
respectively.  Let $\Hneg{m}(\Omega)$ be the dual (with respect to the
$L^2$-inner product) of $H^m(\Omega)\cap H^1_0(\Omega)$ for $m> 0$, with the
norm given by
$$
\nnorm{u}_{\Hneg{m}(\Omega)} = \sup_{v\in H^m(\Omega)\cap H^1_0(\Omega)}\innprd{u}{v}/\nnorm{v}_{H^m(\Omega)}\ .
$$
The space of bounded linear operators on a Banach space $X$ is denoted by $\mathfrak{L}(X)$. 
We regard square $n\times n$ matrices as operators in
$\mathfrak{L}(\mathbb{R}^n)$ and we write matrices and vectors using bold font. If $\vect{A}$ is a symmetric positive definite matrix, 
we denote by $\binnprd{\vect{u}}{\vect{v}}_{\vect{A}}=\vect{v}^T \vect{A} \vect{u}$ 
the $\vect{A}$-dot product of two vectors $\vect{u}, \vect{v}$, and by
$\abs{\vect{u}}_{\vect{A}} = \sqrt{\binnprd{\vect{u}}{\vect{u}}_{\vect{A}}}$ 
the $\vect{A}$-norm; if $\vect{A}=\vect{I}$ we drop the subscript from the inner product
and norm. The space of $m\times n$ matrices is denoted by $M_{m\times n}$; if $m=n$ we write
$M_{n}$ instead of $M_{n\times n}$. Given some norm $\nnorm{\cdot}_s$ on a vector
space $\op{X}$, and $T\in\mathfrak{L}(\op{X})$, we denote by $\nnorm{T}_s$ the induced operator-norm
$$\nnorm{T}_s = \sup_{u\in \op{X},\ \nnorm{u}_s = 1}\nnorm{T u}_s\ .$$
Consequently, if $T\in\mathfrak{L}(L^2(\Omega))$ then $\nnorm{T}$ (no subscripts) is the $L^2$ operator-norm
of~$T$. If $\op{X}$ is a  Hilbert space and $T\in \mathfrak{L}(\op{X})$ then
$T^*\in \mathfrak{L}(\op{X})$ denotes the adjoint of~$T$.
The defining elements of the discrete optimization problem are: the discrete analogues of $\op{K}$, discrete norms,
and discrete inequality constraints, all of which we introduce below.

%
To discretize the optimal control problem~\eqref{eq:reducedoptprob} we consider a sequence of quasi-uniform (in the sense of~\cite{MR2373954})  meshes 
$\op{T}_j,\ j=0,1,2,\dots$, which we assume to be either simplicial (triangular if $d=2$, tetrahedral if $d=3$) 
or rectangular, and let $$h_j=\max\{\mathrm{diam}(T)\ :\ T\in \op{T}_j\}\ ,\ \ j=0,1,2,\dots\ .$$
It is assumed that there are mesh-independent constants $0<\underline{f}\le\overline{f} <1$ (usually $\underline{f}=\overline{f}=1/2$) so that 
$$
\underline{f}\le h_j/h_{j-1}\le \overline{f}\ .
$$
We define the standard finite element spaces: for simplicial elements let
\bes
\op{V}^s_j =  
\{u\in \op{C}(\overline{\Omega})\ :\  \forall T\in \op{T}_j,\ \ u|_T\  \mathrm{is\ linear}, \ \ u|_{\partial \Omega} \equiv 0 \}\ ,
\ees
and for rectangular we use piecewise tensor-products of linear polynomials
\bes
\op{V}^r_j =  
\{u\in \op{C}(\overline{\Omega})\ :\  \forall T\in \op{T}_j,\ \ u|_T\  \in\op{Q}_1, \ \ u|_{\partial \Omega} \equiv 0 \}\ ,
\ees
where $$\op{Q}_1=\left\{\sum_j c_j \prod_{k=1}^d l_{j,k}(x_k)\ : \ l_{j,k}\  \mathrm{linear\  polynomial\  of\  one\  variable}\right\}\ .$$
For simplicity we assume that $\op{T}_{j+1}$ is a uniform refinement of $\op{T}_{j}$ so the associated spaces are nested 
\bes\op{V}_{j} \subset \op{V}_{j+1} \subset H_0^1(\Omega)\ .\ees
Since the algorithms and results are the same for both types of finite element spaces we will denote by $\op{V}_j$
either $\op{V}^s_j$ or $\op{V}^r_j$.
%
Let $N_j=\mathrm{dim}(\op{V}_j)$ and $P^{(j)}_1, \dots, P^{(j)}_{N_j}$ the nodes of $\op{T}_j$ that lie in the interior of 
$\Omega$, and define  $\op{I}_j:\op{C}(\Omega)\rightarrow \op{V}_j$ to  be the standard interpolation operator 
$$\op{I}_j(u)=\sum_{i=1}^{N_j} u(P^{(j)}_i)\varphi_i^{(j)}\ ,$$ 
where $\varphi_i^{(j)}, i=1,\dots, N_j$ are the standard nodal basis functions. 
If we replace exact integration on an element $T$ with vertices $P_1, \dots, P_\nu$ by the cubature
$$
\int_T f(x) dx \approx \frac{\mathrm{vol}(T)}{\nu} \sum_{P\ \mathrm{vertex\ of\ }T} f(P)\ ,
$$
then the $L^2$-inner product is approximated by the mesh-dependent  inner product
\begin{eqnarray*}
\label{def:meship}
\innprd{u}{v}_j = \sum_{i=1}^{N_j} w_i^{(j)}\: u(P^{(j)}_i) v(P^{(j)}_i),\ \ \mathrm{for}\
u, v\in {\mathcal V}_j\ ,
\end{eqnarray*} 
where
\be
\label{eq:weighdef}
w_i^{(j)}{=} \nu^{-1}\sum_{P_i^{(j)}\ \mathrm{vertex\ of\ }T } \mathrm{vol}(T)\ .
\ee
The discrete norms are then given by 
$$\enorm{u}_j\stackrel{\mathrm{def}}{=}\sqrt{\innprd{u}{u}_j}\ .$$
Since the quadrature/cubature is exact for linear functions, or tensor-products of linear functions, respectively, we have
\begin{eqnarray*}
\innprd{u}{v}_j = \int_{\Omega} \op{I}_j (u v),\ \mathrm{for\ all\ }u, v\in\op{V}_j\ .
\end{eqnarray*}
Moreover, due to quasi-uniformity, there exist positive constants $C_1, C_2$
independent of $j\ge 0$ such that
\begin{equation}
\label{eq:equivmeshipl2}
C_1\nnorm{u}\le \enorm{u}_j\le C_2\nnorm{u},\ \forall u \in\op{V}_j\ .
\end{equation}
We should point out that the norm-equivalence~\eqref{eq:equivmeshipl2}
extends to show mesh-independent equivalence of the associated operator-norms.
We say that the weights  $w_i^{(j)}$  are uniform with respect to the mesh $\op{T}_j$ if there exists  $w_j>0$ independent of~$i$  so that
$$w_i^{(j)}=\omega_j h^d\ \  \mathrm{for\  }\  i=1,\dots, N_j\ .$$
We call a  triangulation \emph{locally symmetric} if for every vertex $P$ the associated nodal basis function $\varphi$
is symmetric with respect to the reflection in $P$, that is,
$$\varphi(2P-x) = \varphi(x),\ \ \forall x\in\Omega\ .$$
If a mesh is uniform, then it is locally symmetric and the weights $w_i^{(j)}$  are uniform.

On each space $\op{V}_j$ consider an operator $\op{K}_j\in\mathfrak{L}(\op{V}_j)$ representing a discretization of~$\op{K}$.
For the discrete operators we denote $\op{K}^*_j$ to be the adjoint of $\op{K}_j$ with respect to $\innprd{\cdot}{\cdot}_j$,
that is, $$\innprd{\op{K}_j^* u}{v}_j = \innprd{u}{\op{K}_j v}_j,\ \forall u,v\in\op{V}_j\ .$$
We assume that the operators satisfy the following condition.
\begin{condition}
\label{mgipm:cond:condsmooth} There exists a constant $C=C(\op{K})$ depending on $\op{K}, 
\Omega, \op{T}_{0}$ and independent of $j$ so that the following hold: 
\begin{enumerate}
\item[{\bf [a]}] smoothing:
\begin{equation}
\label{mgipm:cond:par_smooth}
\max(\nnorm{\op{K} u}_{H^m(\Omega)}, \nnorm{\op{K}^* u}_{H^m(\Omega)}) \le C \norm{u},\ \ \forall u\in L^2(\Omega),\  m=0, 1, 2\ ;
\end{equation}
\item[{\bf [b]}] smoothed approximation:
\begin{equation}
\label{mgipm:cond:consist}
\nnorm{\op{K} u - \op{K}_j u}_{H^m(\Omega)}  \le C h_j^{2-m}\norm{u},
\ \ \forall u\in {\mathcal V}_j,\  m=0, 1,\ j\ge 0\ ;
\end{equation}
\item[{\bf [c]}] uniform boundedness of discrete operators and their adjoints:
\begin{equation}
\label{mgipm:cond:unif}
\max(\nnorm{\op{K}^*_j u}_{L^{\infty}(\Omega)}, \nnorm{\op{K}_j u}_{L^{\infty}(\Omega)})  \le C \norm{u},
\ \ \forall u\in {\mathcal V}_j,\  \ j\ge 0\ .
\end{equation}
\end{enumerate}
\end{condition}

%
We now formulate the discrete optimization problem using the discrete norms and we enforce the inequality-constraints 
at the vertices:
\begin{equation}
\label{eq:reducedoptdiscrete}
\left\{\begin{array}{ll}
\vspace{7pt}
\mathrm{minimize} &\op{J}_{j,\beta}({u})\stackrel{\mathrm{def}}{=}\frac{1}{2}\enorm{\op{K}_ju-y_d}_j^2+ \frac{\beta}{2}\enorm{u}_j^2,\ \ \beta>0\ \mathrm{fixed}\\
\vspace{4pt}
\mathrm{subject\ to:\ }& u\in \op{V}_j,\ \  a_i^{(j)}\le u(P_i^{(j)})\le b_i^{(j)} \ \mathrm{for}\ i=1,\dots,N_j\ ,
\end{array}\right .
\end{equation}
where 
the vectors
$\mat{a}^{(j)}=[a_1^{(j)},\dots,a_{N_j}^{(j)}]^T$ (resp.,\mbox{$\mat{b}^{(j)}=[b_1^{(j)},\dots,b_{N_j}^{(j)}]^T$}) represent $L^2$-projections of the functions
$a$ (resp., $b$).
onto~$\op{V}_j$.

If $\mat{K}_j$ is the matrix representation of $\op{K}_j$ in the nodal basis, $\mat{W}_j$ is the diagonal matrix with
diagonal entries $w^{(j)}_1, w^{(j)}_2,\dots, w^{(j)}_{N_j}$, and with $\mat{a}^{(j)}$, $\mat{b}^{(j)}$ defined earlier,
the problem~\eqref{eq:reducedoptdiscrete} reads  in matrix form 
\begin{equation}
\label{eq:reducedoptmatrix}
\left\{\begin{array}{lll}
\vspace{7pt}
\mathrm{minimize} &J_{j,\beta}(\mat{u})\stackrel{\mathrm{def}}{=}
\frac{1}{2}\Abs{\mat{K}_j \mat{u}-\mat{y}_d}_{\mat{W}_j}^2+ \frac{\beta}{2}\Abs{\mat{u}}_{\mat{W}_j}^2,&\ \ \beta>0\ \mathrm{fixed}\\
\vspace{4pt}
\mathrm{subject\ to:\ }& \mat{u}\in \R^{N_j},\  \mat{a}^{(j)}\le \mat{u}\le \mat{b}^{(j)}\ .
\end{array}\right .
\end{equation}
Furthermore, we write $J_{j,\beta}(\mat{u})= \frac{1}{2}\mat{u}^T \mat{C}_j\mat{u} -\mat{f}_j^T\mat{u} + {\bm\gamma}_j$, 
where 
\be
\label{eq:definitionquadform}
\mat{C}_j=\mat{K}_j^T\mat{W}_j\mat{K}_j + \beta \mat{W}_j,\ \ \mat{f}_j=\mat{K}_j^T\mat{W}_j\mat{y}_d,\ \ {\bm \gamma}_j=\frac{1}{2}\mat{y}_d^T \mat{W}_j\mat{y}_d\ .
\ee
We also point out that the adjoint operator $\op{K}_j^*$ is represented by the matrix $\mat{W}_j^{-1}\mat{K}_j^T\mat{W}_j\mat{K}_j $,
so we denote 
$$\mat{K}_j^* \stackrel{\mathrm{def}}{=} \mat{W}_j^{-1}\mat{K}_j^T\mat{W}_j\mat{K}_j\ .$$

\subsection{Optimality conditions and SSNMs}
\label{ssec:ssnm}
Since $\op{J}_{\beta}$ is strictly convex and quadratic,~\eqref{eq:reducedoptprob} has a unique solution $u\in L^2(\Omega)$
satisfying the KKT conditions (e.g. see~\cite{MR2516528, MR2583281}): 
there exist $\lambda_a, \lambda_b\in L^2(\Omega)$ so that
\be
\label{eq:KKTcont}
\left\{\begin{array}{l}
\nabla \op{J}_{\beta} + \lambda_b-\lambda_a =0\ ,\\
a-u\le 0,\ \lambda_a\ge 0,\ (a-u)\lambda_a = 0\ \ \mathrm{a.e.},\\
u-b\le 0,\ \lambda_b\ge 0,\ (u-b)\lambda_b = 0\ \ \mathrm{a.e.}\\
\end{array}
\right .
\ee
After denoting $\lambda=\lambda_a-\lambda_b$, the complementarity system~\eqref{eq:KKTcont} can be written as a nonlinear,
non-smooth system
\be
\label{eq:ssnmcont}
\left\{\begin{array}{l}
\nabla \op{J}_{\beta} -\lambda =0\ ,\\
\lambda-\max(0,\lambda+\sigma (a-u))-\min(0,\lambda+\sigma(b-u)) = 0\ ,
\end{array}
\right .
\ee
where $\sigma>0$ is an arbitrary constant, and $\lambda_a=\max(0,\lambda),\ \lambda_b = -\min(0,\lambda)$. Since 
$\nabla\op{J}_{\beta}(u) = \op{K}^*(\op{K} u - y_d) + \beta u$, for $\sigma=\beta$ the system~\eqref{eq:ssnmcont} is equivalent to
\be 
\label{eq:ssnmcont1}
\nabla \op{J}_{\beta} -\max(0,\op{K}^*(\op{K} u - y_d)+\beta a)-\min(0,\op{K}^*(\op{K} u - y_d)+\beta b) = 0\ .
\ee
Cf.~\cite{MR2839219}, if $a, b\in L^s(\Omega)$ with $s>2$, then Condition~\ref{mgipm:cond:condsmooth} implies that the
nonlinear function in~\eqref{eq:ssnmcont1} is semismooth (slantly differentiable); therefore Newton's method converges 
superlinearly when applied to~\eqref{eq:ssnmcont1}.

To simplify notation, for the remainder of this section we omit the sub- or superscripts  `$j$' if there is no 
danger of confusion; hence  $J_\beta=J_{j,\beta}$, $\vect{C}=\vect{C}_j$, $\vect{a}=\vect{a}^{(j)}$, $N=N_j$, etc.
The KKT and associated semismooth equation for~\eqref{eq:reducedoptmatrix}
is similar to the continuous case, namely the solution  $\mat{u}$ satisfies the following complementarity problem: there exist vectors  
${{\bm \lambda}}_a, {{\bm \lambda}}_b\in\R^{N}$ so that
\be
\label{eq:KKTdisc}
\left\{\begin{array}{l}
\nabla J_{\beta}+{{\bm \lambda}}_b-{{\bm \lambda}}_a =\vect{0}\ ,\\
\vect{a}-\vect{u}\le \vect{0},\ {{\bm \lambda}}_a\ge 0,\ (\vect{a}-\vect{u})\cdot {{\bm \lambda}}_a = \vect{0}\ ,\\
\vect{u}-\vect{b}\le \vect{0},\ {{\bm \lambda}}_b\ge 0,\ (\vect{u}-\vect{b})\cdot {{\bm \lambda}}_b = \vect{0}\ ,\\
\end{array}
\right .
\ee
where $\vect{u}\cdot \vect{v}$ is the componentwise vector multiplication. Since $\nabla J_{\beta}(u)=\vect{C} \mat{u}-\vect{f}$
and $\vect{W}$ is diagonal, after denoting $\hat{{\bm \lambda}}=\vect{W}^{-1}({{\bm \lambda}}_a-{{\bm \lambda}}_b)$ and 
left-multiplying the first equation in~\eqref{eq:KKTdisc} by $\vect{W}^{-1}$, the complementarity system~\eqref{eq:KKTdisc} 
is written as a nonlinear, non-smooth system
\be
\label{eq:ssnmdisc}
\left\{\begin{array}{l}
(\vect{K}^*\mat{K} + \beta \mat{I})\vect{u}-\vect{K}^*\vect{y}_d-\hat{\bm \lambda} =\vect{0}\ ,\\
\hat{\bm \lambda} -\max(\vect{0},\hat{\bm \lambda}+\sigma (\vect{a}-\vect{u}))-\min(\vect{0},\hat{\bm \lambda}+\sigma(\vect{b}-\vect{u})) = \vect{0}\ ,
\end{array}
\right .
\ee
since $\vect{W}^{-1}{\vect{f}} = \vect{K}^*\vect{y}_d$. For $\sigma=\beta$,~\eqref{eq:ssnmdisc} is equivalent to 
\be
\label{eq:ssnmdiscsingle}
\left\{\begin{array}{l}
\hat{\bm \lambda}=(\vect{K}^*\mat{K} + \beta \mat{I})\vect{u}-\vect{K}^*\vect{y}_d\ ,\\
\hat{\bm \lambda} -
\max(\vect{0},\vect{K}^*(\mat{K}\vect{u} -\vect{y}_d)+\beta \vect{a})-\min(\vect{0},\vect{K}^*(\mat{K}\vect{u} -\vect{y}_d)+\beta \vect{b}) = \vect{0}\ .
\end{array}
\right .
\ee
While the resemblance between the discrete system~\eqref{eq:ssnmdiscsingle} and its continuous counterpart~\eqref{eq:ssnmcont1}
extends beyond notation, it is neither evident nor is it the purpose of this paper to prove
mesh-independence of the convergence rate of Newton's method for~\eqref{eq:ssnmdiscsingle}; however, we \emph{do} observe mesh-independence 
in numerical computations even for $\sigma\ne \beta$. We should also note that the mesh-independence results of Hinterm{\"u}ller and 
Ulbrich in~\cite{MR2085262} do not apply directly to~\eqref{eq:ssnmdiscsingle} mainly because here
we discretize the control using continous piecewise linear functions rather than piecewise constants as in~\cite{MR2085262}, so the operator  
$\varphi\mapsto \max(0,\varphi)$ is not well defined in $\op{V}_j$: if $\varphi\in \op{V}_j$ has both negative 
and positive values on an element $T\in\op{T}_j$, then $\max(0,\varphi)$ is no longer in $\op{V}_j$. This is the main reason
for which we resort to a matrix formulation of the discrete optimization problem~\eqref{eq:reducedoptdiscrete} prior to formulating
a discrete semismooth Newton system, instead of formulating  a semismooth Newton system directly for finite element spaces. 
However, our primary interest is to construct multigrid preconditioners for the linear systems arising in the solution process of~\eqref{eq:ssnmdiscsingle}
which we discuss in the next section.

\subsection{Primal-dual formulation and linear systems}
\label{ssec:primaldual}
As shown in~\cite{MR1972219}, the semismooth Newton method for~\eqref{eq:reducedoptmatrix} can be regarded as a primal-dual
active set method which we now describe briefly. If $\vect{u}$, $\hat{\bm \lambda}$ solve~\eqref{eq:ssnmdisc},
then we define the following sets:
\bes
\op{I}&=&\{i\in \{1,\dots,N\}\ :\ \hat{\bm \lambda}_i+\sigma(\vect{a}_i -\vect{u}_i) < 0\ \ \mathrm{and}\ \  \hat{\bm \lambda}_i+\sigma(\vect{b}_i -\vect{u}_i)>0\}\ ,\\
\op{A}^a&=&\{i\in \{1,\dots,N\}\ :\ \hat{\bm \lambda}_i+\sigma(\vect{a}_i -\vect{u}_i) \ge 0\ \ \mathrm{and}\ \  \hat{\bm \lambda}_i+\sigma(\vect{b}_i -\vect{u}_i)>0\}\ ,\\
\op{A}^b&=&\{i\in \{1,\dots,N\}\ :\ \hat{\bm \lambda}_i+\sigma(\vect{a}_i -\vect{u}_i) < 0\ \ \mathrm{and}\ \  \hat{\bm \lambda}_i+\sigma(\vect{b}_i -\vect{u}_i)\le 0\}\ .
\ees
A simple argument shows that if $i\in \op{I}$, then $\vect{a}_i<\vect{u}_i<\vect{b}_i$ and $\hat{\bm \lambda}_i=0$; on the other hand, if 
$i\in \op{A}^a$, then $\vect{a}_i=\vect{u}_i<\vect{b}_i$ and $\hat{\bm \lambda}_i\ge 0$; finally, if  
$i\in \op{A}^b$, then $\vect{a}_i<\vect{u}_i=\vect{b}_i$ and $\hat{\bm \lambda}_i\le 0$. Also, since $\vect{a}_i<\vect{b}_i$ for all $i$, we
have $\op{I}\cup \op{A}^a \cup \op{A}^b = \{1, \dots,N\}$.
The primal-dual active set method produces a sequence of sets $(\op{I}_k, \op{A}^a_k, \op{A}^b_k)_{k=1,2,\dots}$ that approximate
$(\op{I}, \op{A}^a, \op{A}^b)$. Given $(\op{I}_k, \op{A}^a_k, \op{A}^b_k)$, we set the system
\be
\left\{
\begin{array}{l}\vspace{5pt}
  (\vect{K}^*\mat{K} + \beta \mat{I})\vect{u}^{(k+1)}-\hat{\bm \lambda}^{(k+1)} =\vect{K}^*\vect{y}_d\ ,\\\vspace{5pt}
  \vect{u}^{(k+1)}_i=\vect{a}_i,\ \  \mathrm{for}\ \ i\in \op{A}^a_k,\ \ \  \vect{u}^{(k+1)}_i=\vect{b}_i,\ \  \mathrm{for}\ \ i\in \op{A}^b_k\ , \\\vspace{5pt}
  \hat{{\bm \lambda}}^{(k+1)}_i=0,\ \  \mathrm{for}\ \ i\in \op{I}_k\ .
\end{array}
\label{eq:ssnit}
\right .
\ee
The new set-iterates are then given by
\bes
\op{I}_{k+1}&=&\{i\ :\ 
\hat{\bm \lambda}^{(k+1)}_i+\sigma(\vect{a}_i -\vect{u}^{(k+1)}_i) < 0\ \ \mathrm{and}\ \  \hat{\bm \lambda}^{(k+1)}_i+\sigma(\vect{b}_i -\vect{u}^{(k+1)}_i)>0\}\ ,\\
\op{A}^a_{k+1}&=&\{i\ :\ 
\hat{\bm \lambda}^{(k+1)}_i+\sigma(\vect{a}_i -\vect{u}^{(k+1)}_i) \ge 0\ \ \mathrm{and}\ \  \hat{\bm \lambda}^{(k+1)}_i+\sigma(\vect{b}_i -\vect{u}^{(k+1)}_i)>0\}\ ,\\
\op{A}^b_{k+1}&=&\{i\ :\ 
\hat{\bm \lambda}^{(k+1)}_i+\sigma(\vect{a}_i -\vect{u}^{(k+1)}_i) < 0\ \ \mathrm{and}\ \  \hat{\bm \lambda}^{(k+1)}_i+\sigma(\vect{b}_i -\vect{u}^{(k+1)}_i)\le 0\}\ .
\ees
Now consider the splitting  of $\mat{G}=\vect{K}^*\mat{K} + \beta \mat{I}$ according to the sets of indices $\op{I}_k$ and $\op{A}_k=\op{A}^a_k\cup \op{A}^b$.
More precisely, using M{\footnotesize ATLAB} syntax, we define
$$
\mat{G}^{\mathrm{in}, k} = \mat{G}(\op{I}_k,\op{I}_k)\ ,\ \ \mat{G}^{\mathrm{ia},k}=\mat{G}(\op{I}_k,\op{A}_k) ,\ \ \mat{G}^{\mathrm{ai},k}=\mat{G}(\op{A}_k,\op{I}_k),
\ \ \mat{G}^{\mathrm{aa},k}=\mat{G}(\op{A}_k,\op{A}_k)\ .
$$
Since $\vect{u}^{(k+1)}_i$ is determined for $i\in \op{A}_k$, and $\hat{{\bm \lambda}}^{(k+1)}_i=0$ for  $i\in \op{I}_k$,
the equations corresponding to indices from $\op{I}_k$ in the first system of~\eqref{eq:ssnit} reduce to
\be
\label{eq:disclinsysfin1}
\vect{G}^{\mathrm{in},k}\vect{u}^{(k+1)}_{\mathrm{in}} = (\vect{K}^*\vect{y}_d)_{\mathrm{in}}-\vect{G}^{\mathrm{ia},k}\vect{u}^{(k+1)}_{\mathrm{a}}\ ,
\ee
where $\vect{u}^{(k+1)}_{\mathrm{in}} =\vect{u}^{(k+1)}(\op{I}_k)$, $\vect{u}^{(k+1)}_{\mathrm{a}} =\vect{u}^{(k+1)}(\op{A}_k)$,
$(\vect{K}^*\vect{y}_d)_{\mathrm{in}} = (\vect{K}^*\vect{y}_d)(\op{I}_k)$.
After solving~\eqref{eq:disclinsysfin1}, the remaining components of $\hat{{\bm \lambda}}^{(k+1)}$
are identified by
$$
\hat{{\bm \lambda}}^{(k+1)}(\op{A}_{k+1}) = \vect{G}^{\mathrm{ai},k}\vect{u}^{(k+1)}_{\mathrm{in}}+
\vect{G}^{\mathrm{aa},k}\vect{u}^{(k+1)}_{\mathrm{a}} - (\vect{K}^*\vect{y}_d)(\op{A}_{k+1})\ .
$$
Thus the critical step in solving~\eqref{eq:ssnit} is the reduced system~\eqref{eq:disclinsysfin1}.

We should point out that for large-scale problems
the matrices $\vect{K}$, and thus also~$\vect{G}$ and~$\vect{G}^{\mathrm{in},k}$, are expected to be dense and are not formed. 
Therefore we can solve~\eqref{eq:disclinsysfin1} only by means of iterative solvers. Hence, efficient matrix-free preconditioners
are necessary for accelerating the solution process.  Also note that if $\op{A}_k=\emptyset$, 
then the multigrid preconditioners developed in~\cite{MR2429872} can be used.
Our main contribution in this work is the design of two- and multigrid preconditioners 
for~\eqref{eq:disclinsysfin1} for the case when $\op{A}_k\ne\emptyset$.


\subsection{The two-grid preconditioner}
\label{ssec:multigriddef}
Let $j\ge 1$ be a fixed level, to which we shall refer as the fine level. We will first define a two-grid preconditioner for
the matrix $\vect{G}_j^{\mathrm{in},k}$ of~\eqref{eq:disclinsysfin1} that will involve inverting a certain principal minor of $\vect{G}_{j-1}$. 
To design the two-grid  preconditioner we will  regard the matrix $\vect{G}_j^{\mathrm{in},k}$ as an operator between finite element spaces.

\subsubsection{Construction of the two-grid preconditioner}
To simplify notation, consider a fixed index set 
$\op{I}^{(j)}\subseteq\{1,\dots,N_j\}$; this should be regarded as one of the iterates 
$\op{I}_k^{(j)}$ encountered in the solution process described in the previous section.
The linear system we investigate has the form
\be
\label{eq:disclinsysfin}
\vect{G}_j^{\mathrm{in}}\tilde{\vect{u}} = \tilde{\vect{r}}\ ,
\ee
where 
$\vect{G}_j^{\mathrm{in}}=\vect{G}_j(\op{I}^{(j)}, \op{I}^{(j)}), \tilde{\vect{u}}, \tilde{\vect{r}}\in \R^{|\op{I}^{(j)}|}\ .$
We will call vertex $P_i$ or an index $i$ \emph{inactive} if $i\in \op{I}^{(j)}$.
Define the \emph{fine inactive space} by
$$\op{V}_j^{\mathrm{in}} = \mathrm{span}\{\varphi_i^{(j)}\ :\ i\in \op{I}^{(j)}\}\ ,$$
and the \emph{fine inactive domain} by
\be
\label{eq:coarseinactdomain}
\Omega_j^{\mathrm{in}} = \bigcup_{i\in \op{I}^{(j)}}\mathrm{supp}(\varphi_i^{(j)})\ ,
\ee
where $\mathrm{supp}(u)$ is the support of the function $u$.
The critical component of the preconditioner is the definition of the \emph{coarse inactive index set}: 
\be
\label{eq:coarseinactdef}
\op{I}^{(j-1)}\stackrel{\mathrm{def}}{=}\{i\in \{1,\dots, N_{j-1}\}\ :\ \mathrm{supp}(\varphi_i^{(j-1)})\subseteq \Omega_j^{\mathrm{in}}\}\ .
\ee
To be more precise, if $i_f$ is the index in the fine numbering associated to the coarse index $i_c$, that is, $P^{(j)}_{i_f} = P^{(j-1)}_{i_c}$, 
the definition above is equivalent to saying  that $i_c\in \op{I}^{(j-1)}$ if $i_f$ together with all its neighbouring \emph{fine} indices are
inactive. In Figure~\ref{fig:inactiveset} we depict a set of inactive fine nodes by filled circles, and the associated coarse inactive nodes
 by hollow circles. The coarse nodes that are not inactive are shown with a square hollow marker.
The \emph{coarse inactive space} is now defined to be
\be
\label{eq:coarseinactspace}
\op{V}_{j-1}^{\mathrm{in}} = \mathrm{span}\{\varphi_i^{(j-1)}\ :\ i\in \op{I}^{(j-1)}\}\ ,
\ee
and the \emph{coarse inactive domain}  is given by
$$\Omega_{j-1}^{\mathrm{in}} = \bigcup_{i\in \op{I}^{(j-1)}}\mathrm{supp}(\varphi_i^{(j-1)})\ .$$
Note that $\op{V}^{\mathrm{in}}_{j-1}\subseteq \op{V}^{\mathrm{in}}_{j}$ and $\Omega_{j-1}^{\mathrm{in}}\subseteq\Omega_j^{\mathrm{in}}$.
In connection with $\Omega_{j}^{\mathrm{in}}$ we also 
define \emph{numerical interior} $\mathrm{Int}_{\mathrm{n}} \Omega_{j}^{\mathrm{in}}$ of $\Omega_{j}^{\mathrm{in}}$ (relative to $\Omega_{j-1}^{\mathrm{in}}$)
to be the  union of all \emph{coarse} elements $T$ included in the fine inactive domain, that is, $T\subseteq \Omega_{j}^{\mathrm{in}}$, and whose vertices 
are either in $\op{I}^{(j-1)}$ or lie on the boundary of $\Omega$ (see Figure~\ref{fig:inactiveset}). Furthermore,
let the \emph{numerical boundary} of $\Omega_{j}^{\mathrm{in}}$ (relative to $\Omega_{j-1}^{\mathrm{in}}$) be given by
$$\partial_{\mathrm{n}}\Omega_{j}^{\mathrm{in}} = \Omega_{j}^{\mathrm{in}}\setminus \mathrm{Int}_{\mathrm{n}} \Omega_{h}^{\mathrm{in}}\ .$$
Note that $\mathrm{Int}_{\mathrm{n}} \Omega_{j}^{\mathrm{in}}\subseteq \Omega_{j-1}^{\mathrm{in}}$. 

\begin{figure}[!htb]
\begin{center}
        \includegraphics[width=9in]{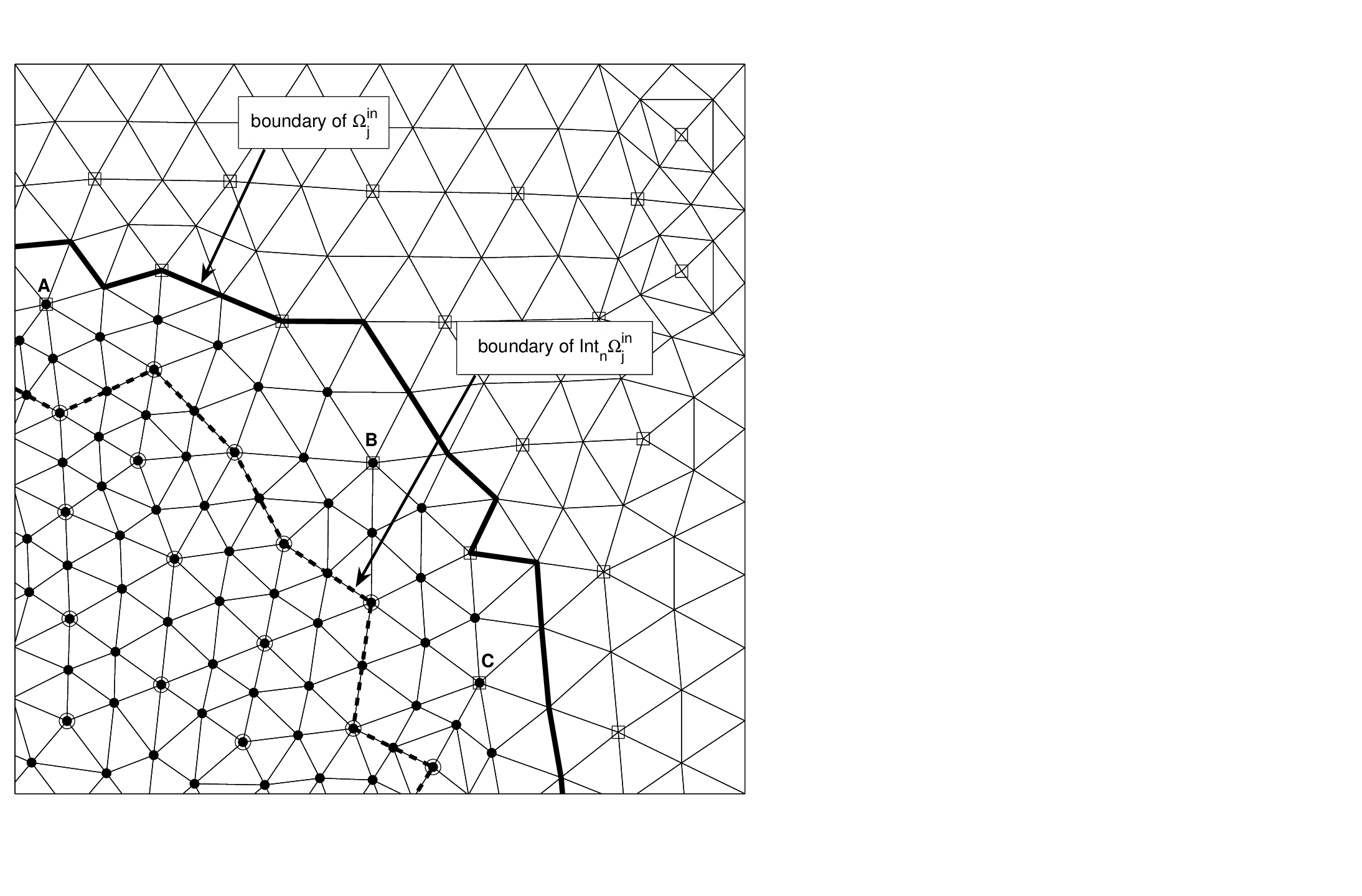}
\caption{Inactive fine nodes are marked with a dot, inactive coarse nodes are marked with a circle, and active coarse nodes are marked
with a square. Note that the coarse nodes $A, B, C$ are active because the supports of the corresponding nodal basis functions are not included in
$\Omega_j^{\mathrm{in}}$, even though the nodes themselves lie in the interior of $\Omega_j^{\mathrm{in}}$. The area between the solid and the dotted 
lines lies in the numerical boundary
of  $\Omega_j^{\mathrm{in}}$.} 
\label{fig:inactiveset}
\end{center}
\end{figure}

Let now $\op{W}^{\mathrm{in}}_{j-1}$ be the $L^2$-orthogonal complement of $\op{V}_{j-1}^{\mathrm{in}}$ in
$\op{V}_j^{\mathrm{in}}$ and define the $L^2$-projectors
$$\pi^{\mathrm{in}}_{j-1}:\op{V}^{\mathrm{in}}_j\to \op{V}^{\mathrm{in}}_{j-1}\ ,\ \ \ \ 
\rho^{\mathrm{in}}_{j-1}:\op{V}^{\mathrm{in}}_j\to \op{W}^{\mathrm{in}}_{j-1}\ ,$$ 
so $\pi_{j-1}^{\mathrm{in}}+\rho_{j-1}^{\mathrm{in}}$ is the identity on
$\op{V}_j^{\mathrm{in}}$.  Furthermore,  let 
$\op{E}_j^{\mathrm{in}}: \op{V}_j^{\mathrm{in}} \to \op{V}_j$ be the natural embedding obtained by extending a function with zero outside of
$\Omega_j^{\mathrm{in}}$.  We may
occasionally omit the explicit use of the embedding operators to ease notation.  We also define the
restriction $R_{j-1}:\op{V}_{j}\to\op{V}_{j-1}$ as the adjoint with
respect to $\innprd{\cdot}{\cdot}_j$ to the embedding of $\op{V}_{j-1}$
in $\op{V}_j$, that is, for $u\in \op{V}_j$
\be
\label{eq:restrdef}
\innprd{u}{v}_j = \innprd{R_{j-1}u}{v}_{j-1},\ \ \forall v\in \op{V}_{j-1}\ ,
\ee
and let $P_{j}^{\mathrm{in}}: \op{V}_j\to \op{V}_j^{\mathrm{in}}$
the projection with respect to $\innprd{\cdot}{\cdot}_j$ given by
$$P_j^{\mathrm{in}} \left(\sum_{i=1}^{N_j} u_i \varphi_i^{(j)}\right) = \sum_{i\in \op{I}^{(j)}} u_i \varphi_i^{(j)}\ .$$
Furthermore, denote by $\pi_{j}$ the orthogonal $L^2$-projector onto the space $\op{V}_j$.
From the equivalence~\eqref{eq:equivmeshipl2} of the discrete norms with the $L^2$-norm it follows that
\be
\label{eq:reststab}
\nnorm{R_{j}u}\le C\nnorm{u}\ ,\ j=0,1,\dots,\ .
\ee
for some mesh-independent constant $C$.

The matrix $\vect{G}_j^{\mathrm{in}}$ in~\eqref{eq:disclinsysfin}  represents the operator
\be
\label{eq:Gopdef}
\op{G}_j^{\mathrm{in}} = P_j^{\mathrm{in}}\left(\op{K}_j^*\op{K}_j + \beta I\right)\op{E}_j^{\mathrm{in}}\ ,
\ee
and thus matrix-vector products for $\vect{G}_j^{\mathrm{in}}$ are computed accordingly. We define a \emph{two-grid preconditioner} $M_j^{\mathrm{in}}$ 
for $\op{G}_j^{\mathrm{in}}$ as in  Dr{\u a}g{\u a}nescu and Dupont~\cite{MR2429872} for the unconstrained case: 
\be
\label{eq:tgprecdef}
\op{M}_j^{\mathrm{in}} = \overbrace{P_{j-1}^{\mathrm{in}}\left(\op{K}_{j-1}^*\op{K}_{j-1}+\beta I\right)\op{E}_{j-1}^{\mathrm{in}}}^{\op{G}_{j-1}^{\mathrm{in}}}
  \pi_{j-1}^{\mathrm{in}} + \beta \rho_{j-1}^{\mathrm{in}}\ .
\ee
Note that the inverse $\op{S}_j^{\mathrm{in}}$  of the preconditioner $M_j^{\mathrm{in}}$ has the explicit form
\be
\label{eq:tgprecdefinv}
\op{S}_j^{\mathrm{in}}\stackrel{\mathrm{def}}{=}
(\op{M}_j^{\mathrm{in}})^{-1} = \left(\op{G}_{j-1}^{\mathrm{in}}\right)^{-1}\pi_{j-1}^{\mathrm{in}} + 
\beta^{-1} \rho_{j-1}^{\mathrm{in}}\ .
\ee
Since none of the matrices representing $\op{G}_j^{\mathrm{in}}$ or $\op{M}_j^{\mathrm{in}}$ are formed, it is the operator~$\op{S}_j^{\mathrm{in}}$ that we 
need to apply in practice, so the explicit formula~\eqref{eq:tgprecdefinv} for $(\op{M}_j^{\mathrm{in}})^{-1} $ is essential. In light of
this fact we should remark that, if  the projection~$\pi_{j-1}^{\mathrm{in}}$ were to be replaced by a restriction operator, as is the case in 
classical multigrid, then~\eqref{eq:tgprecdefinv}  would no longer hold.

\subsubsection{Matrix-form of the preconditioner}
\label{sssec:matrixform}
In order to describe the matrix-form of the preconditioner in M{\footnotesize ATLAB} form we introduce the mass matrix~$\mat{L}_j$ on $\op{V}_j$, and we denote by 
$\mat{J}_{j}\in M_{N_{j}\times N_{j-1}}$ the matrix representing the interpolation operator $\op{I}_{j}|_{\op{V}_{j-1}}\in
\mathfrak{L}(\op{V}_{j-1}, \op{V}_j)$. 
We assume that the inactive indices on level $j$
are stored  in the  vector $\mat{i}_j$,  and define the matrices
\bes
\mat{L}^I_j = \mat{L}_{j}(\mat{i}_{j},\mat{i}_{j}),\ \ \mat{J}^I_{j} = \mat{J}_j(\mat{i}_{j},\mat{i}_{j-1}),\
\ \ \mat{E}_{j}^I = \mat{I}(:,\mat{i}_{j}),\ \ 
\mat{P}_{j}^I = (\mat{E}_j^I)^T,
\ees
where we used M{\footnotesize ATLAB} syntax for the selection of submatrices.
Note that $\mat{E}_{j}^I$ represents the extension operator $\op{E}_{j}^{\mathrm{in}}$ and $\mat{P}_{j}^I$ the operator $P_j^{\mathrm{in}}$. We can now write the 
projector-operator $\pi_{j-1}^{\mathrm{in}}$ in matrix-form as
$${\bm \Pi}_{j-1}^{I} = \left(\mat{L}^I_{j-1}\right)^{-1}\cdot 
(\mat{J}^I_{j})^T\cdot \mat{L}^I_j\ ,$$
and $\rho_{j-1}^{\mathrm{in}}$ is represented by $(\mat{I}-\mat{J}^I_{j}{\bm \Pi}_{j-1}^{I})$. So $\op{M}_j^{\mathrm{in}} $ is represented by
the matrix 
$$
\mat{M}_j^I = \mat{J}_{j}^I\overbrace{\mat{P}_{j-1}^I(\mat{K}_{j-1}^T \mat{K}_{j-1}+ \beta \mat{I})\mat{E}_{j-1}^I}^{\mathrm{represents}\ \op{G}_{j-1}^{\mathrm{in}}}
{\bm \Pi}_{j-1}^{I}+\beta (\mat{I}-\mat{J}^I_{j}{\bm \Pi}_{j-1}^{I})\ ,
$$
and $\op{S}_j^{\mathrm{in}}$ is represented by
\begin{equation}
\label{eq:matformprec}
\mat{S}_j^I = \mat{J}_{j}^I\left(\mat{P}_{j-1}^I(\mat{K}_{j-1}^T \mat{K}_{j-1}+ \beta \mat{I})\mat{E}_{j-1}^I\right)^{-1}
{\bm \Pi}_{j-1}^{I}+\beta^{-1} (\mat{I}-\mat{J}^I_{j}{\bm \Pi}_{j-1}^{I})\ .
\end{equation}
We should point out that, due to the presence of ${\bm \Pi}_{j-1}^{I}$, the matrices $\mat{M}_j^I$ and $\mat{S}_j^I$ are slightly nonsymmetric, hence one
has to employ solvers for nonsymmetric systems in connection with the $\mat{M}_j^I $ preconditioner. We found that conjugate gradient squared (CGS)
is quite efficient (see Section~\ref{sec:numerics}).

A symmetric alterative to $\mat{S}_j^I$ is obtained by replacing the orthogonal projection ${\bm \Pi}_{j-1}^{I}$ in~\eqref{eq:matformprec} with 
the matrix ${\bm R}_{j-1}^I= {\bm R}_{j-1}(\mat{i}_{j-1},\mat{i}_{j}),$ 
where the matrix ${\bm R}_{j-1}$ represents the restriction operator $R_{j-1}$ defined in~\eqref{eq:restrdef}. This gives rise to the
preconditioner
\begin{equation}
\label{eq:matformprecsym}
\widehat{\mat{S}}_j^I = \mat{J}_{j}^I\left(\mat{P}_{j-1}^I(\mat{K}_{j-1}^T \mat{K}_{j-1}+ \beta \mat{I})\mat{E}_{j-1}^I\right)^{-1}
{\bm R}_{j-1}^{I}+\beta^{-1} (\mat{I}-\mat{J}^I_{j}{\bm R}_{j-1}^{I})\ .
\end{equation}
It can be easily verified that $\widehat{\mat{S}}_j^I$ is symmetric, hence it can be used in practice as a preconditioner to conjugate gradient (CG). We conduct some of
our numerical experiments in Section~\ref{sec:numerics} using multigrid preconditioners defived from $\widehat{\mat{S}}_j^I$  rather than from ${\mat{S}}_j^I$.
Note that for a uniform grid in $d$ dimensions we have ${\bm R}_{j-1}= 2^{-d}{\bm J}_{j}^T$.

\subsubsection{Spectral distance estimation}
To quantify the quality of the two-grid preconditioner $\op{M}_j^{\mathrm{in}}$ we will estimate the \emph{spectral distance}  $d_{\sigma}$ 
between $(\op{G}_j^{\mathrm{in}})^{-1}$ and its inverse $\op{S}_j^{\mathrm{in}}$ given by~\eqref{eq:tgprecdefinv}. The use of the spectral distance ensures that such estimates
extend automatically to multigrid preconditioners, as discussed in Section~\ref{ssec:multigrid} and Section~\ref{sec:mganalysis} (see also~\cite{MR2429872, Dra:Pet:ipm}). 
We briefly recall the definition of the
spectral distance, as introduced in~\cite{MR2429872}. Given a Hilbert space $(\op{X}, \innprd{\cdot}{\cdot})$ we denote by 
$\mathfrak{L}_+(\op{X})$ the set of operators with positive definite symmetric part:
$$ \mathfrak{L}_+(\op{X})= \left\{T\in \mathfrak{L}(\op{X}): \innprd{T u}{u} > 0,\ \ \forall u\in
\op{X}\setminus\{0\}\right\}\ .$$  
Let the joined numerical range of $S, T\in\mathfrak{L}_+(\op{X})$ be given by
\begin{eqnarray*}
W(S, T) = \left\{ \frac{\innprd{S_{\mathbb C} w}{w}}{\innprd{T_{\mathbb C} w}{w}}\ :\ w \in \op{X}^{\mathbb{C}}\setminus \{0\}\right\}\ ,
\end{eqnarray*}
where $T_{\mathbb C}(u+{\bf i}v) = T(u)+{\bf i}T(v)$ is the complexification of $T$.
The spectral distance between $S, T \in \mathfrak{L}_+(\op{X})$, is a measure of spectral equivalence between $S$ and $T$, and it is defined by
\begin{eqnarray*}
d_{\sigma}(S, T)& =& \sup\{\abs{\ln z}\ :\  z \in W(S, T) \}\ ,
\end{eqnarray*}
where $\ln$ is the branch of the logarithm corresponding to 
$\mathbb{C}\setminus (-\infty, 0]$.
Following Lemma 3.2 in~\cite{MR2429872}, if $W(S,T)\subseteq \op{B}_{\alpha}(1) = \{z\in\mathbb{C}\ :\ \abs{z-1} < \alpha\}$ with $\alpha\in (0,1)$, then
\begin{equation}
\label{eq:distineq}
d_{\sigma}(S, T)\le \frac{\abs{\ln (1-\alpha)}}{\alpha} \sup\{|z-1|\  : \ z\in W(S,T)\}\ ,
\end{equation}
which offers a practical way to estimate the spectral distance when it is small.
The spectral distance serves 
both as a means to quantify the quality of a preconditioner
and also as a convenient analysis tool for multigrid algorithms. Essentially, if two operators $S, T$ satisfy
$$1-\delta\le \Abs{\frac{\innprd{S_{\mathbb C} w}{w}}{\innprd{T_{\mathbb C} w}{w}}}\le 1+ \delta\ ,\  \ \forall w \in \op{X}^{\mathbb{C}}\setminus \{0\},$$
with $\delta \ll 1$, then $d_{\sigma}(S, T) \approx \delta$. 
If $N\approx G^{-1}$ is a preconditioner for $G$, then both $d_{\sigma}(N, G^{-1})$ and $d_{\sigma}(N^{-1}, G)$ (quantities which are
equal if $G, N$ are symmetric) are shown to control the spectral radius $\rho(I-N G)$ (see~\cite{Dra:Soa:mgstokes})
which is an accepted quality-measure for a preconditioner. The advantage of using $d_{\sigma}$ over $\rho(I-N G)$ is that the former is a true distance function.

The  main result of this article is the following theorem.
\begin{theorem}
\label{maintheorem}
If the operators $\op{K}$ and $\op{K}_j$ satisfy Condition~\textnormal{\ref{mgipm:cond:condsmooth}} and the weights  $w_i^{(j)}$  are uniform, 
then there exists $\delta>0$ and a constant $C(\op{K})$~\textnormal{(}see Condition~\textnormal{\ref{mgipm:cond:condsmooth}}~\textnormal{)} 
independent of $j$ and of the inactive set so that
\be
\label{eq:mainestimate}
d_{\sigma}\left(\left(\op{G}_j^{\mathrm{in}}\right)^{-1}, \op{S}_j^{\mathrm{in}}\right) \le C \beta^{-1}\left(h_j^2 + \sqrt{\mu_{j}^{\mathrm{in}}}\right) \ ,
\ee 
where $\mu_{j}^{\mathrm{in}}$ is the Lebesgue measure of $\partial_{\mathrm{n}}\Omega_{j}^{\mathrm{in}}$, provided that 
$$\beta^{-1}\left(h_j^2 + \sqrt{\mu_{j}^{\mathrm{in}}}\right)<\delta\ .$$
\end{theorem}
We postpone the proof of Theorem~\ref{maintheorem}  until Section~\ref{sec:analysis}.
\begin{remark}
\label{rem:sqrtest}
The natural question arises as to how to estimate $\mu_{j}^{\mathrm{in}}$. In the  worst case scenario
there are no coarse inactive nodes, so $\Omega_{j-1}^{\mathrm{in}}=\emptyset$; 
therefore $\partial_{\mathrm{n}}\Omega_{j}^{\mathrm{in}} = \Omega$, case
in which the two-grid preconditioner is $\beta I$, so essentially there is no preconditioner.
However, if $u$ is the solution of~\eqref{eq:reducedoptprob} and 
the continuous inactive set defined by~\mbox{$\Omega^{\mathrm{in}}=\{x\in \Omega\ : \ u(x)>0\}$} is 
a domain with Lipschitz boundary, then the discrete inactive set 
$\Omega_j^{\mathrm{in}}$ is expected to be  close to $\Omega^{\mathrm{in}}$ provided that
a good initial guess at the inactive set is available. In this case we expect that
$\partial_{\mathrm{n}}\Omega_{j}^{\mathrm{in}}$ will lie within $C h_j$ of the topological
boundary of $\Omega^{\mathrm{in}}$, therefore 
$$\mu_{j}^{\mathrm{in}}\approx C h_j\ ,$$
where $C$ is proportional to  the $(d-1)$-dimensional measure of $\partial\Omega^{\mathrm{in}}$.
Hence, the estimate~\eqref{eq:mainestimate} truly implies 
\be
\label{eq:mainestimatesqrt}
d_{\sigma}\left(\left(\op{G}_j^{\mathrm{in}}\right)^{-1}, \op{S}_j^{\mathrm{in}}\right) \le C \frac{\sqrt{h}_j}{\beta}\ ,
\ee 
which is consistent with the numerical experiments in Section~\ref{sec:numerics}. Under certain, reasonable
assumptions we will show that the estimate~\eqref{eq:mainestimatesqrt} extends to multilevel preconditioners.
\end{remark}

In case the grid is quasi-uniform but not uniform we apply Theorem~\ref{maintheorem} to the matrix
$\widetilde{\mat{K}}_j \stackrel{\mathrm{def}}{=} \mat{W}_j \mat{K}_j$.
The important aspect in the estimate 
 is the verification of Condition~\textnormal{\ref{mgipm:cond:condsmooth}} by
$\widetilde{\mat{K}}_j$. Following~\cite{Dra:Pet:ipm} we  introduce the following indirect measure of grid-smoothness:
for each $j$  we consider a $C^2$-function $w_j:\overline{\Omega}\to\R$ so that
\be
\label{eq:wjdef}
w_j(P_i^{(j)}) = w_i^{(j)},\ \ \forall i=1,2,\dots, N_j,
\ee
with $w_i^{(j)}$ given by~\eqref{eq:weighdef}. With this notation the matrix $\widetilde{\mat{K}}_j$ represents the
operator~\mbox{$\widetilde{\op{K}}_j\in \mathfrak{L}(\op{V}_j)$} defined by
\be
\label{eq:ktildef}
\widetilde{\op{K}}_j u\stackrel{\mathrm{def}}{=}\op{I}_j(w_j\cdot (\op{K}_j u))\ .
\ee
Note that, because the grids are hierarchical ($\op{T}_{j}$ is obtained from $\op{T}_{j-1}$ by adding nodes),
the function $w_j$ can serve for defining all operators $\widetilde{\op{K}}_l$, for $l=0, 1, \dots, j$.
Consider $j$ fixed, and define $\widetilde{\op{K}}\in\mathfrak{L}(L^2(\Omega))$ by
$$
\widetilde{\op{K}} u\stackrel{\mathrm{def}}{=}w_j\cdot (\op{K}_j u)\ .
$$
By Proposition~4.8 in~\cite{Dra:Pet:ipm}, the operators $(\widetilde{\op{K}}_l)_{l=0,1,\dots, j}$ together with
the continuous operator $\widetilde{\op{K}}$ satisfy Condition~\ref{mgipm:cond:condsmooth}
with 
$$C(\widetilde{\op{K}}) = \nnorm{w_{j}}_{W_2^{\infty}(\Omega)} C(\op{K})\ .$$
Thus we establish the following corollary.
\begin{corollary}
\label{cor:nonunifgrid}
If the operators $\op{K}$ and $\op{K}_j$ satisfy Condition~\textnormal{\ref{mgipm:cond:condsmooth}} and $w_j\in C^2(\overline{\Omega})$ satisfies
\eqref{eq:wjdef}, then there exists $\delta>0$ and a constant 
$C(\op{K})$ independent of $j$ and the inactive set so that
\be
\label{eq:estimatenonunif}
d_{\sigma}\left(\left(\op{G}_j^{\mathrm{in}}\right)^{-1}, \op{S}_j^{\mathrm{in}}\right) \le C \beta^{-1}
\nnorm{w_{j}}_{W_2^{\infty}(\Omega)} \left(h_j + \sqrt{\mu_{j}^{\mathrm{in}}}\right) \ ,
\ee 
where $\mu_{j}^{\mathrm{in}}$ is as in Theorem~$\ref{maintheorem}$, provided that 
$$\beta^{-1}\nnorm{w_{j}}_{W_2^{\infty}(\Omega)}\left(h_j + \sqrt{\mu_{j}^{\mathrm{in}}}\right)<\delta\ .$$
\end{corollary}
Corollary~\ref{cor:nonunifgrid} marks a decrease of the power of $h_j$ from two to one
compared to the estimate in  Theorem~\ref{maintheorem}. This fact is due to the nonuniformity of the grid
as can be seen in the analysis of Section~\ref{sec:analysis}. However, in general, the larger of the two terms in~\eqref{eq:estimatenonunif}
is expected to be $\sqrt{\mu_{j}^{\mathrm{in}}}$, so the main effect of the nonuniformity on the estimate is the presence of 
the factor~$\nnorm{w_{j}}_{W_2^{\infty}(\Omega)}$.

\subsection{The multigrid preconditioner}
\label{ssec:multigrid}
We now assume the level $j_{\max}\ge 0$ to be fixed;  we will refer to $j_{\max}$ as the finest level. 
The goal is to construct a multigrid operator
$\op{Z}_{j_{\max}}^{\mathrm{in}}\approx(\op{G}_{j_{\max}}^{\mathrm{in}})^{-1}$ which satisfies the following conditions: 
(i) if $j_{\max}=1$ then $\op{Z}_{1}^{\mathrm{in}} = \op{S}_1^{\mathrm{in}}$;
(ii) an estimate like~\eqref{eq:mainestimate} holds if we replace $\op{S}_{j_{\max}}^{\mathrm{in}}$ with $\op{Z}_{j_{\max}}^{\mathrm{in}}$.
In order to construct $\op{Z}_{j_{\max}}^{\mathrm{in}}$ we must first specify
the coarser inactive domains $\Omega_j^{\mathrm{in}}$, inactive index-sets
$\op{I}^{(j)}$, and inactive spaces $\op{V}_j^{\mathrm{in}}$ for $j=j_{\max}-1, \dots, 0$. All these entities are defined recursively
using~\eqref{eq:coarseinactdomain},~\eqref{eq:coarseinactdef},~\eqref{eq:coarseinactspace}, and  are essentially specified
by the sets $\op{I}^{(j)}$. Hence, we give below the algorithm for computing the inactive-index sets $\op{I}^{(j)}$
for \mbox{$j=j_{\max}-1,\dots,0$} (recall that~$\op{I}^{(j)}$ is one of the set-iterates in the SSNM solution process). 
Given a vertex~$P^{(j-1)}_{i_c}$ of the triangulation $\op{T}_{j-1}$, let $P^{(j)}_{i_f}$ be its fine label. 
We define the `fine neighborhood' of $P^{(j-1)}_{i_c}$ by
$$\op{N}_j(P^{(j-1)}_{i_c}) = \{R^{(j)}\: :\: R^{(j)}\ \mathrm{neighbor\  in\ } \op{T}_j\  \mathrm{of}\  P^{(j)}_{i_f} \}\cup \{P^{(j)}_{i_f}\}\ .$$
\noindent 

In principle, the algorithm for finding the coarse inactive sets (see Algorithm~\ref{alg:inactive_sets} below) 
can be easily implemented using M{\footnotesize ATLAB}'s set operations. 
However, we should point out that for medium- and  large-scale problems Algorithm~\ref{alg:inactive_sets} should be implemented using a divide-and-conquer strategy 
in order to avoid repeated memory allocation at step 5.
\begin{algorithm}[Inactive set definition]
\label{alg:inactive_sets} 
\begin{enumerate}
\item[$1.$] {\tt for}\hspace{3pt} $j=j_{\max}\::\:-1\::\:1$ \hspace{3pt} 
\vspace{6pt}
\item[$2.$] \hspace{20pt} $\op{I}^{(j-1)}=\emptyset$
\vspace{6pt}
\item[$3.$] \hspace{20pt} {\tt for} $i=1\::\:N_{j-1}$
\vspace{6pt}
\item[$4.$] \hspace{40pt}  {\tt if}\hspace{10pt} $\op{N}_j(P_i^{(j-1)})\subseteq \op{I}^{(j)}$
\vspace{6pt}
\item[$5.$] \hspace{60pt}  $\op{I}^{(j-1)}=\op{I}^{(j-1)}\cup \{P_i^{(j-1)}\}$
\end{enumerate}\end{algorithm}
\vspace{10pt}

If we define the operator $$\mathfrak{I}_{j-1}^j:\mathfrak{L}(\op{V}_{j-1}^{\mathrm{in}}) \to \mathfrak{L}(\op{V}_{j}^{\mathrm{in}}),\ \ 
\mathfrak{I}_{j-1}^j(\op{X}) = \op{X}\cdot \pi_j^{\mathrm{in}} + \beta^{-1}(I-\pi_j^{\mathrm{in}})\ ,$$
then $\op{S}_j^{\mathrm{in}}$ can be written as
\be
\label{eq:sdefop}
\op{S}_j^{\mathrm{in}} = \mathfrak{I}_{j-1}^j\left((\op{G}_{j-1}^{\mathrm{in}})^{-1}\right)\ .
\ee
In light of~\eqref{eq:sdefop} and the continuity of the affine operator $\mathfrak{I}_{j-1}^j$, it is 
tempting to define recursively the following multigrid preconditioner: 
\be
\label{eq:badmgprecdef} 
\widetilde{\op{Z}}_{j}^{\mathrm{in}}
=\left\{
\begin{array}{ccl}\vspace{10pt}
(\op{G}_{j}^{\mathrm{in}})^{-1}& , & \mathrm{if}\ j=0\ ,\\
\mathfrak{I}_{j-1}^j(\widetilde{\op{Z}}_{j-1}^{\mathrm{in}})& , & \mathrm{if}\ j\ge 1\ .
\end{array}
\right .
\ee
As shown in~\cite{MR2429872}, the $V$-cycle type preconditioner $\widetilde{\op{Z}}_{j}^{\mathrm{in}}$
does not satisfy condition~(ii) above. In fact one can show that, under the conditions set in Remark~\ref{rem:sqrtest},  
$\widetilde{\op{Z}}_{j}^{\mathrm{in}}$ satisfies~\eqref{eq:mainestimatesqrt} with $h_j$ replaced by $h_0$. 
As a result, the number of preconditioned iterations would no longer be decreasing with $h_j\downarrow 0$, as is the case for the two-grid preconditioner, 
but would be bounded. Instead of the definition~\eqref{eq:badmgprecdef}, 
we employ the same strategy adopted in~\cite{MR2429872, Dra:Pet:ipm}, which guarantees that the estimate for the multigrid preconditioner will be similar to
that  for the two-grid preconditioner (except for a constant factor). In order to do so, we define the operator $\mathfrak{N}_j$ by
$$\mathfrak{N}_j:\mathfrak{L}(\op{V}_{j}^{\mathrm{in}})\to \mathfrak{L}(\op{V}_{j}^{\mathrm{in}}),
\ \ \mathfrak{N}_j(\op{X})\stackrel{\mathrm{def}}{=} 2\op{X}-\op{X}\cdot\op{G}_j^{\mathrm{in}}\cdot\op{X}\ .$$
Note that $\mathfrak{N}_j$ is the Newton iterator for the operator-equation
$$\op{X}^{-1}-\op{G}_j^{\mathrm{in}}=0\ ,$$
as shown by Dr{\u a}g{\u a}nescu and Dupont in~\cite{MR2429872}. This implies that, if $\op{X}_0$ is a good 
approximation of $(\op{G}_j^{\mathrm{in}})^{-1}$, then $\op{X}_1 = \mathfrak{N}_j(\op{X}_0)$ is the first Newton iterate of the above 
operator-equation starting with $\op{X}_0$, and so $\op{X}_1$ is  significantly closer to 
$(\op{G}_j^{\mathrm{in}})^{-1}$ than~$\op{X}_0$. This idea was also used in~\cite{Dra:Pet:ipm} 
to construct multigrid preconditioners of the same quality as the two-grid preconditioners.
%
The Algorithm~\ref{alg:multigrid} essentially has a W-cycle structure;  we construct recursively a sequence of coarser-level operators~$\op{Z}_{j}$
for $0\le j\le j_{\max}-1$ so that each  application $\op{Z}_{j}$ requires one application of~$\op{G}^{\mathrm{in}}_j$. At the finest level no
application of $\op{G}^{\mathrm{in}}_{j_{\max}}$ is performed inside the preconditioner~$\op{Z}^{\mathrm{in}}_{j_{\max}}$.
\begin{algorithm}[Operator-form definition of $\op{Z}_{j_{\max}}^{\mathrm{in}}$]
\label{alg:multigrid}
\begin{enumerate}
\item[$1.$] {\tt if}\hspace{3pt} $j=0$ \hspace{3pt} 
\vspace{6pt}
\item[$2.$] \hspace{20pt} $\op{Z}^{\mathrm{in}}_0 := (\op{G}^{\mathrm{in}}_0)^{-1}$ \hspace{88pt} {\tt \%} coarsest level
\vspace{6pt}
\item[$3.$] {\tt else if} \hspace{3pt} $j<j_{\max}$ \hspace{10pt} 
\vspace{6pt}
\item[$4.$] \hspace{42pt} $\op{Z}^{\mathrm{in}}_j :=\mathfrak{N}_j(\mathfrak{I}_{j-1}^j (\op{Z}^{\mathrm{in}}_{j-1}))$ \hspace{29pt} {\tt \%} intermediate level
\vspace{6pt}
\item[$5.$]  \hspace{26pt} {\tt else}\hspace{129pt} {\tt \%} here $j=j_{\max}$
\vspace{6pt}
\item[$6.$] \hspace{42pt} $\op{Z}^{\mathrm{in}}_j :=\mathfrak{I}_{j-1}^j (\op{Z}^{\mathrm{in}}_{j-1})$ \hspace{50pt} {\tt \%} finest level
\end{enumerate}
\end{algorithm}
\vspace{10pt}

We give a detailed analysis of the multigrid preconditioner in Section~\ref{sec:mganalysis}. 
To anticipate, essentially we show that, under conditions set forth in Theorem~\ref{thm:mganalysis}, the operator $\op{Z}^{\mathrm{in}}_{j_{\max}}$
satisfies an estimate  like~\eqref{eq:mainestimate}.

We conclude our description of the multigrid preconditioner by commenting on its computational cost. 
First, it goes without saying that for large-scale problems the operator~$\op{Z}_{j_{\max}}^{\mathrm{in}}$ is not formed;
rather, its action on a vector is implemented as a multilevel, matrix-free iteration. 
This matrix-free implementation,  described  in~\cite{Dra:Pet:ipm}, naturally calls for a matrix-free 
implementation of $\op{G}^{\mathrm{in}}_j$ which essentially relies on applying $\op{G}_j$ matrix-free, as
follows from~\eqref{eq:Gopdef}. The computational cost of the application of the multigrid preconditioner
is estimated in~\cite{Dra:Pet:ipm}; for completeness we briefly describe the 
result. Let $C(j)$ denote the cost of applying the operator~$\op{G}_j$ and assume that there exists $0<\tau<1/2$
so that $C(j-1) \le \tau C(j)$. For example, if the cost of applying $\op{K}_j$ is linear in the number of variables,
as is the case if $\op{K}_j$ is computed using classical multigrid, then for  a problem with $d$ spatial dimensions
we expect $\tau=2^{-d}$. We also assume that unpreconditioned CG solves the problem~\eqref{eq:disclinsysfin}
in at most $K_{cg}$ iterations to acceptable precision, with $K_{cg}$ independent of $j$. This assumption is certainly justified by
the uniform-boundedness of the condition number of $\op{G}_j$ (see also Theorem~3.1 in~\cite{stollwathen}); however, due to the compactness  of the 
operators, $K_{cg}$ is in fact relatively small (see Section~\ref{ssec:twodimnum} for numerical results).
If $T(j)$ denotes the cost of applying $\op{Z}_{j_{\max}}^{\mathrm{in}}$, then cf.~\cite{Dra:Pet:ipm}
we have
\be
\label{eq:costest}
T(j_{\max})\le \tau C(j_{\max})\left(K_{cg} (2\tau)^{j_{\max}-1}+\frac{1-(2\tau)^{j_{\max}-1}}{1-2\tau} \right)\ .
\ee
For example, if four levels  are used ($j_{\max}=3$), $\tau=2^{-3}$, and with $K_{cg} =20$, then 
$T(j_{\max})\le 0.3125 \cdot C(j_{\max})$; if 5  levels are used, then $T(j_{\max})\le 0.2031 \cdot C(j_{\max})$.
The estimate~\eqref{eq:costest} shows that for large-scale,
multidimensional problems where $\tau$ is expected to be relatively small, the cost of applying the
preconditioner is a fraction of the cost of computing a matrix-vector multiplication at the finest level
if multiple levels are used. For $j_{\max}\gg 1$, $T(j_{\max})\le (\tau/(1-\tau)) \: C(j_{\max})$.

\section{Analysis of the two-grid preconditioner}
\label{sec:analysis}
The main step in the analysis is to evaluate the norm-distance between the operators $\op{G}_j^{\mathrm{in}}$ and $\op{M}_j^{\mathrm{in}}$ which is done
in Proposition~\ref{prop:twogridprec}. The plan of the analysis generally resembles that of the analysis of multigrid preconditioners for
IPMs from~\cite{Dra:Pet:ipm}, however, certain critical estimates related to the projection $\pi_j^{\mathrm{in}}$ are different
for the case of SSNMs.

First we restate Lemma 4.3 in~\cite{Dra:Pet:ipm} as
\begin{lemma}
\label{lma:innprdapprox}
With $(w_i^{(j)})_{1\le i\le N_j}$ chosen as in~{\textnormal{\eqref{eq:weighdef}}} there exists a constant $C$ dependent on $\op{T}_0$ and  independent of 
$j$ so that
\begin{equation}
\label{eq:innprdtsapprox}
\abs{\innprd{u}{v}_j-\innprd{u}{v}}\le C h_j^2 \nnorm{u}_{H^1(\Omega)}\cdot\nnorm{v}_{H^1(\Omega)}\ ,\ \forall u,v \in \op{V}_j.
\end{equation}
\end{lemma}
\noindent We also recall Lemma 4.4 in~\cite{Dra:Pet:ipm}:
%
%
\begin{lemma}
\label{lma:manyresults}
If $\op{K}, \op{K}_j$ satisfy Condition~$\ref{mgipm:cond:condsmooth}$, there exist 
constants $C(\op{K})$ and \mbox{$C'=C'(\Omega)$}  independent of $j$ such that the following hold:\\
\textnormal{(a)} $H^1, L^2$ - uniform stability of $\op{K}_j$: 
\begin{equation}
\label{eq:H1stabdisc}
\nnorm{\op{K}_j u}_{H^m(\Omega)}\le C(\op{K}) \nnorm{u}\ ,\ \forall u\in \op{V}_j,\ m=0, 1,\ \ j=0,1,\dots;
\end{equation}
\textnormal{(b)} smoothing of negative-index norm: 
\begin{equation}
\label{eq:Hm2stabcont}
\nnorm{\op{K} u} \le C(\op{K}) \norm{u}_{\Hneg{m}}\ ,\ \forall u\in {\mathcal V}_j, m=1,2\ ;
\end{equation}
\textnormal{(c)}  negative-index norm approximation of the identity by $\pi_{j-1}, \op{R}_{j-1}$:
\begin{eqnarray}
\label{eq:projid_approx}
\norm{(I - \pi_{j-1}) u}_{\Hneg{2}(\Omega)}& \le& C' h_j^2 \norm{u},\ \ \forall u\in \op{V}_j; \\
\label{eq:restr_approx}
\norm{(I - \op{R}_{j-1}) u}_{\Hneg{p}(\Omega)}& \le& C' h_j^p \norm{u},\ \ \forall u\in \op{V}_j\ , 
\end{eqnarray}
where $p=1$ on a quasi-uniform grid, and  $p=2$ on a locally symmetric grid;\\
\textnormal{(d)} $\op{K}$ diminishes high-frequencies:
\begin{eqnarray}
\label{eq:projid_approxL}
\norm{\op{K}(I - \pi_{j-1}) u}& \le& C(\op{K}) h_j^2 \norm{u},\ \ \forall u\in \op{V}_j\ ; \\
\label{eq:restr_approxK}
\norm{\op{K}(I - \op{R}_{j-1}) u}& \le& C(\op{K}) h_j^p \norm{u},\ \ \forall u\in \op{V}_j\ ,
\end{eqnarray}
where $p=1$ on an unstructured grid, and  $p=2$ on a locally symmetric grid;\\
\textnormal{(e)} 
\begin{eqnarray}
\label{eq:innprdKKH}
\abs{\innprd{\op{K} u}{\op{K} v}-\innprd{\op{K}_j u}{\op{K}_j v}} \le C(\op{K}) h_j^2 \nnorm{u}\cdot \nnorm{v}\ , \ \forall u, v\in \op{V}_j\ .
\end{eqnarray}
\end{lemma}
The main difference between the two-grid preconditioners for the unconstrained case vs. the constrained case is in 
the properties of the projectors on the coarse spaces. While the projector $\pi_{j-1}$ on the entire
coarse space $\op{V}_{j-1}$ satisfies~\eqref{eq:projid_approx}, the projector on the inactive space
$\op{V}^{\mathrm{in}}_{j-1}$ satisfies the weaker estimate below.
\begin{lemma}
\label{lma:l2projapprox}
If $d\le 3$ there exists a constant $C$ depending on the domain $\Omega$ and the base triangulation
$\op{T}_0$ so that
\be
\label{eq:l2projapprox}
\nnorm{(I-\pi^{\mathrm{in}}_{j-1}) u}_{\Hneg{2}(\Omega)} 
\le C\left(h_j^2 + \sqrt{\mu_{j}^{\mathrm{in}}}\right) \nnorm{u},\ \ \mathrm{for\ all}\ u\in \op{V}^{\mathrm{in}}_{j}\ ,
\ee
where $\mu_{j}^{\mathrm{in}}$ is the Lebesgue measure of $\partial_{\mathrm{n}}\Omega_{j}^{\mathrm{in}}$. 
The constant $C$ is independent of $j$ and of the inactive set $\Omega^{\mathrm{in}}_j$.
\end{lemma}
\begin{proof}
Let $u\in \op{V}_j^{\mathrm{in}}$, $v\in H^2(\Omega)\cap H_0^1(\Omega)$ be arbitrary, and let $v_{j-1}\in \op{V}_{j-1}^{\mathrm{in}}$ be the
natural interpolant of $v$ in $\op{V}_{j-1}^{\mathrm{in}}$, that is
$$v_{j-1} =\sum_{i\in\op{I}^{(j-1)}} v(P_i^{(j-1)})\varphi_i^{(j-1)}\ .$$
Let $T\in\op{T}_{j-1}$ be a coarse element lying in $\Omega_{j}^{\mathrm{in}}$.
If  $T\subseteq\mathrm{Int}_{\mathrm{n}}\Omega_{j}^{\mathrm{in}}$ then
$v_{j-1}$ agrees on $T$ with the interpolant of $v$ in $\op{V}_{j-1}$. Therefore a standard interpolation 
estimate (see~\cite{MR2373954}) applied on $\mathrm{Int}_n \Omega_{j}^{\mathrm{in}}$ gives
\bes
\nnorm{v-v_{j-1}}_{L^2(\mathrm{Int}_n \Omega_{j}^{\mathrm{in}})}\le  C h_{j-1}^2 \abs{v}_{H^2(\mathrm{Int}_n \Omega_{j}^{\mathrm{in}})}
\le C\underline{f}^{-1} h_{j}^2 \abs{v}_{H^2(\Omega)}\ . 
\ees
On a coarse element $T$ that satisfies $T\subseteq \partial_{\mathrm{n}}\Omega_{j}^{\mathrm{in}}$, if $v_{j-1}$ is not identically zero on~$T$
then $v_{j-1}$ and $v$ agree at least for one vertex of $T$ and potentially
disagree at a vertex corresponding to an active constraint. In either case the bound
\mbox{$\nnorm{v_{j-1}}_{L^{\infty}(T)}\le \nnorm{v}_{L^{\infty}(T)}$} holds.
Hence,
\bes
\nnorm{v-v_{j-1}}_{L^2(\partial_{\mathrm{n}}\Omega_{j}^{\mathrm{in}})}\le 2 
\sqrt{\mu_{j}^{\mathrm{in}}}\:\nnorm{v}_{L^{\infty}(\partial_{\mathrm{n}}\Omega_{j}^{\mathrm{in}})}\le C \sqrt{\mu_{j}^{\mathrm{in}}} \:\nnorm{v}_{H^2(\Omega)}\ ,
\ees
by  Sobolev's inequality. We have
\be
\nonumber
\lefteqn{\Abs{\innprd{(I-\pi^{\mathrm{in}}_{j-1})u}{v}}}\\
 &=& \Abs{\innprd{(I-\pi^{\mathrm{in}}_{j-1})u}{v-v_{j-1}} } = \Abs{\int_{\Omega_j^{\mathrm{in}}}(u-\pi^{\mathrm{in}}_{j-1} u)\:(v-v_{j-1})}\\
\label{eq:lemmapproxproj}
&\le&\Abs{\int_{\mathrm{Int}_{\mathrm n} \Omega_j^{\mathrm{in}}}(u-\pi^{\mathrm{in}}_{j-1} u)\:(v-v_{j-1})}+
\Abs{\int_{\partial_{\mathrm{n}} \Omega_j^{\mathrm{in}}}(u-\pi^{\mathrm{in}}_{j-1} u)\:(v-v_{j-1})}\\
&\le& C \nnorm{u-\pi^{\mathrm{in}}_{j-1} u}_{L^2(\Omega_{j}^{\mathrm{in}})}\left(h_j^2 + \sqrt{\mu_{j}^{\mathrm{in}}}\right)\nnorm{v}_{H^2(\Omega)}\ .
\ee
 Since $\nnorm{\pi^{\mathrm{in}}_{j-1} u}\le \nnorm{u}$, the result now follows after
dividing by $\nnorm{v}_{H^2(\Omega)}$ and taking the supremum over all $v\in H^2(\Omega)$.
\end{proof}
%
\begin{proposition}
\label{prop:twogridprec}
If the operators $\op{K}$ and $\op{K}_j$ satisfy Condition~\textnormal{\ref{mgipm:cond:condsmooth}} and the weights  $w_i^{(j)}$  are uniform, 
then there exists a constant $C$ independent on $j$ and the inactive set so that
\begin{equation}
\label{eq:twogridprecnormest}
\nnorm{\op{G}_j^{\mathrm{in}} - \op{M}_j^{\mathrm{in}}} \le C \left(h_j^2+ \sqrt{\mu_{j}^{\mathrm{in}}}\right) \ ,
\end{equation}
where $\mu_{j}^{\mathrm{in}}$ is the Lebesgue measure of $\partial_{\mathrm{n}}\Omega_{j}^{\mathrm{in}}$.
\end{proposition}
\begin{proof}
Since $P_j^{\mathrm{in}} \op{E}_{j}^{\mathrm{in}}=\pi_{j-1}^{\mathrm{in}}+\rho_{j-1}^{\mathrm{in}}=I_{\op{V}_j^{\mathrm{in}}}$, we have 
\be
\label{eq:diffGM}
\op{G}_j^{\mathrm{in}} - \op{M}_j^{\mathrm{in}}  =
 P_j^{\mathrm{in}} \op{K}_j^*\op{K}_j \op{E}_{j}^{\mathrm{in}}- P_{j-1}^{\mathrm{in}}\op{K}_{j-1}^*\op{K}_{j-1}\op{E}_{j-1}^{\mathrm{in}}\pi_{j-1}^{\mathrm{in}}\ .
\ee
The argument has a structure that is similar to the proof of Proposition~4.5 in~\cite{Dra:Pet:ipm} with  changes due to the specific 
approximation properties of $P_{j-1}^{\mathrm{in}}$ and $\pi_{j-1}^{\mathrm{in}}$.
Let~\mbox{$u,v\in \op{V}_j^{\mathrm{in}}$} be arbitrary, and define $\widetilde{u}  =\op{E}_{j-1}^{\mathrm{in}}\pi^{\mathrm{in}}_{j-1} u$. 
Note that $\nnorm{\widetilde{u}}\le \nnorm{u}$.
We first examine the difference
\bes
\lefteqn{\abs{\innprd{\op{K}_{j-1}^*\op{K}_{j-1}\widetilde{u}}{v}_j - \innprd{P_{j-1}^{\mathrm{in}}\op{K}_{j-1}^*\op{K}_{j-1}\widetilde{u}}{v}_j} }\\
&=&\Abs{\int_{\Omega_j^{\mathrm{in}}}(I- P_{j-1}^{\mathrm{in}})(\op{K}_{j-1}^*\op{K}_{j-1}\widetilde{u})\cdot v} =
\Abs{\int_{\partial_{\mathrm{n}}\Omega_j^{\mathrm{in}}}(I- P_{j-1}^{\mathrm{in}})(\op{K}_{j-1}^*\op{K}_{j-1}\widetilde{u})\cdot v}\\
&\le&
\nnorm{(I- P_{j-1}^{\mathrm{in}})(\op{K}_{j-1}^*\op{K}_{j-1}\widetilde{u})}_{L^2(\partial_{\mathrm{n}}\Omega_j^{\mathrm{in}})}\cdot 
\nnorm{v}_{L^2(\partial_{\mathrm{n}}\Omega_j^{\mathrm{in}})}\\
&\le& \sqrt{\mu_{j}^{\mathrm{in}}} \nnorm{(I- P_{j-1}^{\mathrm{in}})(\op{K}_{j-1}^*\op{K}_{j-1}\widetilde{u})}_{L^{\infty}(\partial_{\mathrm{n}}\Omega_j^{\mathrm{in}})}
\cdot \nnorm{v}_{L^2(\partial_{\mathrm{n}}\Omega_j^{\mathrm{in}})}\ ,
\ees
where we used that $(I- P_{j-1}^{\mathrm{in}})(\op{K}_{j-1}^*\op{K}_{j-1}\widetilde{u})$ is zero in $\mathrm{Int}_n\Omega_j^{\mathrm{in}}$ and that 
$v$ is zero outside of $\Omega_j^{\mathrm{in}}$.
Since for $w\in \op{V}_{j-1}$  and $x\in\Omega$ the function value $(I- P_{j-1}^{\mathrm{in}}) w(x)$ lies between $0$ and $w(x)$ we have
\bes
\nnorm{(I- P_{j-1}^{\mathrm{in}})(\op{K}_{j-1}^*\op{K}_{j-1}\widetilde{u})}_{L^{\infty}(\partial_{\mathrm{n}}\Omega_j^{\mathrm{in}})}\le
\nnorm{\op{K}_{j-1}^*\op{K}_{j-1}\widetilde{u}}_{L^{\infty}(\Omega_j)} 
\stackrel{\eqref{mgipm:cond:unif}}{\le} C \nnorm{\widetilde{u}} \le C \nnorm{u}\ .
\ees
Therefore 
\be
\label{eq:Pinest}
\abs{\innprd{\op{K}_{j-1}^*\op{K}_{j-1}\widetilde{u}}{v}_j - \innprd{P_{j-1}^{\mathrm{in}}\op{K}_{j-1}^*\op{K}_{j-1}\widetilde{u}}{v}_j} \le 
C \sqrt{\mu_{j}^{\mathrm{in}}}\:\nnorm{u}\cdot\nnorm{v}\ .
\ee
We return to estimating $(\op{G}_j^{\mathrm{in}} - \op{M}_j^{\mathrm{in}})$. We have
\bes
\lefteqn{\innprd{(\op{G}_j^{\mathrm{in}} - \op{M}_j^{\mathrm{in}})u}{v}_j }\\ 
&=&\innprd{P_j^{\mathrm{in}} \op{K}_j^*\op{K}_j \op{E}_{j}^{\mathrm{in}} u}{v}_j - \innprd{\op{K}_{j-1}^*\op{K}_{j-1}\op{E}_{j-1}^{\mathrm{in}}\pi^{\mathrm{in}}_{j-1} u}{v}_j\\
&&+\overbrace{\innprd{\op{K}_{j-1}^*\op{K}_{j-1}\op{E}_{j-1}^{\mathrm{in}}\pi^{\mathrm{in}}_{j-1} u}{v}_j - 
\innprd{P_{j-1}^{\mathrm{in}}\op{K}_{j-1}^*\op{K}_{j-1}\op{E}_{j-1}^{\mathrm{in}}\pi^{\mathrm{in}}_{j-1} u}{v}_j}^{A_0}\\
&\stackrel{\eqref{eq:restrdef}}{=}&
\innprd{\op{K}_j^*\op{K}_j \op{E}_{j}^{\mathrm{in}} u}{ v}_j - 
\innprd{\op{K}_{j-1}^*\op{K}_{j-1}\op{E}_{j-1}^{\mathrm{in}}\pi^{\mathrm{in}}_{j-1} u}{R_{j-1}v}_{j-1} + A_0\\
&=&\innprd{\op{K}_j\op{E}_{j}^{\mathrm{in}} u}{\op{K}_j \op{E}_{j}^{\mathrm{in}} v}_j - 
\innprd{\op{K}_{j-1}\op{E}_{j-1}^{\mathrm{in}}\pi^{\mathrm{in}}_{j-1} u}{\op{K}_{j-1}\op{E}_{j-1}^{\mathrm{in}} R_{j-1}v}_{j-1}
+ A_0\\
&=&A_0+ \overbrace{\innprd{\op{K}_j\op{E}_{j}^{\mathrm{in}} u}{\op{K}_j \op{E}_{j}^{\mathrm{in}} v}_j - 
\innprd{\op{K}_j\op{E}_{j}^{\mathrm{in}} u}{\op{K}_j \op{E}_{j}^{\mathrm{in}} v}}^{A_1}+\\
&&\overbrace{\innprd{\op{K}_j\op{E}_{j}^{\mathrm{in}} u}{\op{K}_j \op{E}_{j}^{\mathrm{in}} v} - 
\innprd{\op{K}\op{E}_{j}^{\mathrm{in}} u}{\op{K}\op{E}_{j}^{\mathrm{in}} v}}^{A_2}+\\
&&\overbrace{\innprd{\op{K}\op{E}_{j}^{\mathrm{in}} u}{\op{K}\op{E}_{j}^{\mathrm{in}} v} - 
\innprd{\op{K}\op{E}_{j-1}^{\mathrm{in}}\pi^{\mathrm{in}}_{j-1} u}{\op{K}\op{E}_{j-1}^{\mathrm{in}}  R_{j-1}v}}^{A_3}+\\
&&\overbrace{\innprd{\op{K}\op{E}_{j-1}^{\mathrm{in}}\pi^{\mathrm{in}}_{j-1} u}{\op{K}\op{E}_{j-1}^{\mathrm{in}}  R_{j-1}v}
- \innprd{\op{K}_{j-1}\op{E}_{j-1}^{\mathrm{in}}\pi^{\mathrm{in}}_{j-1} u}{\op{K}_{j-1}\op{E}_{j-1}^{\mathrm{in}}  R_{j-1}v}}^{A_4}+\\
&&\overbrace{\innprd{\op{K}_{j-1}\op{E}_{j-1}^{\mathrm{in}}\pi^{\mathrm{in}}_{j-1} u}{\op{K}_{j-1}\op{E}_{j-1}^{\mathrm{in}}  R_{j-1}v}- 
\innprd{\op{K}_{j-1}\op{E}_{j-1}^{\mathrm{in}}\pi^{\mathrm{in}}_{j-1} u}{\op{K}_{j-1}\op{E}_{j-1}^{\mathrm{in}}  R_{j-1}v}_{j-1}}^{A_5}\ .
\ees
By~\eqref{eq:Pinest}
$$\abs{A_0}\le C\sqrt{\mu_{j}^{\mathrm{in}}}\:\nnorm{u}\cdot\nnorm{v}\ .$$
For $A_1$ and $A_5$ we use~\eqref{eq:innprdtsapprox},~\eqref{eq:H1stabdisc}, and~\eqref{eq:reststab} to conclude that
\bes
\abs{A_1}&\le& C h_j^2 \nnorm{\op{K}_j\op{E}_{j}^{\mathrm{in}} u}_{H^1(\Omega)}\cdot \nnorm{\op{K}_j\op{E}_{j}^{\mathrm{in}} v}_{H^1(\Omega)}\le 
 C h_j^2 \nnorm{u}\cdot\nnorm{v}\ ,\\
\abs{A_5}&\le& C h_{j-1}^2 \nnorm{\op{K}_{j-1}\op{E}_{j-1}^{\mathrm{in}} \pi_{j-1}^{\mathrm{in}} u}_{H^1(\Omega)}\cdot 
\nnorm{\op{K}_j\op{E}_{j}^{\mathrm{in}}   R_{j-1}v}_{H^1(\Omega)} \\
&\le & C \underline{f}^{2} h_j^2 \nnorm{\pi_{j-1}^{\mathrm{in}} u}\cdot\nnorm{R_{j-1}v}\le C h_j^2 \nnorm{u}\cdot\nnorm{v} \ .
\ees
Estimation of $A_2, A_4$, essentially involving~\eqref{eq:innprdKKH}, is handled exactly as in~\cite{Dra:Pet:ipm} to give
\bes
\max(\abs{A_2},\abs{A_4})\le C h_j^2 \nnorm{u}\cdot\nnorm{v} \ .
\ees
Finally the term $A_3$ is estimated by
\bes
\abs{A_3}&\le& \abs{\innprd{\op{K}\op{E}_{j}^{\mathrm{in}}(I -\pi^{\mathrm{in}}_{j-1}) u}{\op{K}\op{E}_{j}^{\mathrm{in}} v}} + 
\abs{\innprd{\op{K}\op{E}_{j-1}^{\mathrm{in}}\pi^{\mathrm{in}}_{j-1} u}{(\op{K}\op{E}_{j}^{\mathrm{in}}(I - R_{j-1})v}}\\
&\le& C \left(h_j^2+\sqrt{\mu_{j}^{\mathrm{in}}}\right)\nnorm{u}\cdot\nnorm{v} \ .
\ees
We conclude that
\bes
\Abs{\innprd{(\op{G}_j^{\mathrm{in}} - \op{M}_j^{\mathrm{in}})u}{v}_j}\le C \left(h_j^2+\sqrt{\mu_{j}^{\mathrm{in}}}\right)
\nnorm{u}\cdot\nnorm{v},\ \forall u, v\in \op{V}^{\mathrm{in}}_j\ , 
\ees
hence $\enorm{\op{G}_j^{\mathrm{in}} - \op{M}_j^{\mathrm{in}}}_j\le \left(h_j^2+\sqrt{\mu_{j}^{\mathrm{in}}}\right)$, and the result follows from the 
equivalence of the operator-norms $\enorm{\cdot}_j,\ \nnorm{\cdot}$.
\end{proof}

Theorem~\ref{maintheorem} now follows easily from Proposition~\ref{prop:twogridprec}:
\begin{proof}
First note that for $u\in\op{V}_j^{\mathrm{in}}$ 
\be
\innprd{\op{G}_j^{\mathrm{in}}u}{u}_j = \innprd{\op{K}_j^{\mathrm{in}}\op{E}^{\mathrm{in}}u}{\op{K}_j^{\mathrm{in}}\op{E}^{\mathrm{in}}u}_j+
\beta\innprd{u}{u}_j \ge \beta\innprd{u}{u}_j\ ,
\ee
so $\sigma(\op{G}_j^{\mathrm{in}}) \subseteq[\beta,\infty)$, which implies, due to the symmetry of  
$\op{G}_j^{\mathrm{in}}$ with respect to $\innprd{\cdot}{\cdot}_j$ and the equivalence of the operator-norms $\enorm{\cdot}_j,\ \nnorm{\cdot}$, that 
$$\nnorm{\left(\op{G}_j^{\mathrm{in}}\right)^{-\frac{1}{2}}} \le C \beta^{-\frac{1}{2}}\ .$$
By Proposition~\ref{prop:twogridprec}
$$\nnorm{I-\left(\op{G}_j^{\mathrm{in}}\right)^{-\frac{1}{2}} \op{M}_j^{\mathrm{in}} \left(\op{G}_j^{\mathrm{in}}\right)^{-\frac{1}{2}}}\le
C\nnorm{\left(\op{G}_j^{\mathrm{in}}\right)^{-\frac{1}{2}}}^2\cdot\nnorm{\op{G}_j^{\mathrm{in}}-\op{M}_j^{\mathrm{in}}}\le
C \beta^{-1}\left(h_j^2+\sqrt{\mu_{j}^{\mathrm{in}}}\right)\ .$$
The remainder of the argument proceeds as in the proof of Theorem 4.9 in~\cite{Dra:Pet:ipm}.
\end{proof}

%
%
\section{Analysis of the multigrid preconditioner}
\label{sec:mganalysis}
We begin with a few technical results concerning the spectral distance.
\begin{lemma}
\label{lma:techlma1}
If $w_1, w_2 \in \H=\{z\in \C\ :\ \Re(z)>0\}$, then
\be
\label{eq:w12xineq}
\Abs{\ln\frac{w_1+x}{w_2+x}} \le \Abs{\ln\frac{w_1}{w_2}},\ \ \forall x\ge 0\ ,
\ee
where $\Re(z)$ denotes the real part of a complex number $z$. 
\end{lemma}

We postpone  the proof of Lemma~\ref{lma:techlma1} until Appendix~\ref{sec:specdist}.
The next result shows that $\mathfrak{I}_{j-1}^j$ is non-expansive when measured in the spectral distance.
\begin{lemma}
\label{lma:nonexp}
For all $j=1,2,...$, and $\op{G},\op{H}\in \mathfrak{L}_+(\op{V}_{j-1})$
\be
\label{eq:nonexp}
d_{\sigma}(\mathfrak{I}_{j-1}^j(\op{G}),\mathfrak{I}_{j-1}^j(\op{H}))\le d_{\sigma}(\op{G},\op{H})\ ,
\ee
where the spectral distance is computed using the $L^2$-inner product.
\end{lemma}
\begin{proof}
Let $u\in \op{V}_{j}^{\C}$ be arbitrary so that $\pi_{j-1}^{\mathrm{in}}u\ne 0$. Then 
\bes
\lefteqn{\Abs{\ln\frac{\innprd{\op{G}\pi_{j-1}^{\mathrm{in}}u + \beta^{-1} \rho_{j-1}^{\mathrm{in}}u}{u}}{\innprd{\op{H}\pi_{j-1}^{\mathrm{in}}u + 
	\beta^{-1} \rho_{j-1}^{\mathrm{in}}u}{u}}}}\\
&=&
\Abs{\ln\frac{\innprd{\op{G}\pi_{j-1}^{\mathrm{in}}u}{\pi_{j-1}^{\mathrm{in}}u} + \beta^{-1} \innprd{\rho_{j-1}^{\mathrm{in}}u}{\rho_{j-1}^{\mathrm{in}}u}}
  {\innprd{\op{H}\pi_{j-1}^{\mathrm{in}}u}{\pi_{j-1}^{\mathrm{in}}u} + \beta^{-1} \innprd{\rho_{j-1}^{\mathrm{in}}u}{\rho_{j-1}^{\mathrm{in}}u}}}
\stackrel{\eqref{eq:w12xineq}}{\le}
\Abs{\ln\frac{\innprd{\op{G}\pi_{j-1}^{\mathrm{in}}u}{\pi_{j-1}^{\mathrm{in}}u}}
  {\innprd{\op{H}\pi_{j-1}^{\mathrm{in}}u}{\pi_{j-1}^{\mathrm{in}}u}}}\\
&\le & \sup_{v\in \op{V}_{j}^{\C}\setminus \{0\}}
\Abs{\ln\frac{\innprd{\op{G}v}{v}}
  {\innprd{\op{H}v}{v}}} = d_{\sigma}(\op{G}, \op{H})\ .
\ees
The above holds trivially if $\pi_{j-1}^{\mathrm{in}}u = 0$ and $u\ne 0$, so~\eqref{eq:nonexp} follows by taking the supremum over all
$u\in \op{V}_{j}^{\C}\setminus \{0\}$.
\end{proof}

Note that for symmetric operators Lemma~\ref{lma:nonexp} reduces to Lemma~5.1 in~\cite{MR2429872}. However, as mentioned earlier,
neither $\op{G}_j^{\mathrm{in}}$ nor $\op{M}_j^{\mathrm{in}}$ are expected to be symmetric with respect to
the $L^2$-inner product, hence the need for Lemma~\ref{lma:nonexp}.

\begin{lemma}
\label{lma:sqconvnewt}
If the operators $\op{K}, \op{K}_j$ satisfy Condition~$\ref{mgipm:cond:condsmooth}$ with constant $C$, then there exist $\delta>0$ and $C_1=C_1(\beta, C)$
so that, if $\op{X}\in\mathfrak{L}(\op{V}^{\mathrm{in}}_j)$ satisfies $d_{\sigma}(\op{X},(\op{G}^{\mathrm{in}}_j)^{-1})< \delta$, then
\be
\label{eq:sqconvnewt}
d_{\sigma}(\mathfrak{N}_j(\op{X}),(\op{G}^{\mathrm{in}}_j)^{-1})\le C_1 \left(d_{\sigma}(\op{X},(\op{G}^{\mathrm{in}}_j)^{-1})\right)^2\ .
\ee
\end{lemma}
The proof of Lemma~\ref{lma:sqconvnewt} is similar to that of Theorem~3.12 in~\cite{MR2429872}, except that here none of the operators involved
is symmetric with respect to the $L^2$-inner product. Unfortunately, the lack of appropriate analysis tools for the non-symmetric case
(see also the comment following Lemma~\ref{lma:sqnonsymest})
allow for a weaker than desired result, with $C_1$ depending on $\beta$ instead of being universal. For completeness we give the proof of 
Lemma~\ref{lma:sqconvnewt} in Appendix~\ref{sec:specdist}.

We recall Lemma~5.3 from~\cite{MR2429872} as
\begin{lemma}
\label{eq:lma53}
Let $(e_j)_{j\ge 0}$ and $(a_j)_{j \ge 0}$ be positive numbers satisfying
\begin{equation}
\label{eq:mg_err_abst}
e_{j} \le C(e_{j-1} + a_{j})^2\ ,\ \  a_{j}  \le  a_{j-1} \le f^{-1} a_{j},\ \ j=1, 2, \dots, 
\end{equation}
for some $0<f<1$. If $a_0 \le \frac{f}{4 C}$ and if $e_0 \le 4 C a_0^2$, then 
\begin{equation}
e_j \le 4 C a_j^2,\ \ \forall j> 0\ .
\end{equation}  
\end{lemma}
\begin{theorem}
\label{thm:mganalysis}
Assume that the operators $\op{K}$ and $\op{K}_j$ satisfy Condition~\textnormal{\ref{mgipm:cond:condsmooth}} and the weights  $w_i^{(j)}$  are uniform. 
Given an index $j_{\max}\ge 1$ let the index-sets~\mbox{$\op{I}^{(j)}$}, \mbox{$0\le j \le j_{\max}$}, be given as 
in Algorithm~$\ref{alg:inactive_sets}$, and consider the corresponding inactive domains $\Omega_{j}^{\mathrm{in}}$. If 
$\mu^{\mathrm{in}}_j$ is the Lebesgue measure of $\partial_{\mathrm{n}}\Omega_{j}^{\mathrm{in}}$, it is assumed that
there exists $f\in (1/4,1)$ so that $\mu^{\mathrm{in}}_j\le \mu^{\mathrm{in}}_{j-1}\le f^{-2} \mu^{\mathrm{in}}_j$
for $j=1,2, \dots, j_{\max}$  and that \mbox{$\beta^{-1}(h_0^2+\sqrt{\mu_0^{\mathrm{in}}})\le f/(4 C_1)$} with $C_1$ 
as in Lemma~$\ref{lma:sqconvnewt}$.
Then there exists a constant $C_2=C_2(C_1, \beta)$ so that
\be
\label{eq:multigridest}
d_{\sigma}\left(\op{Z}^{\mathrm{in}}_{j_{\max}}, (\op{G}^{\mathrm{in}}_{j_{\max}})^{-1} \right) \le 
C_2\beta^{-1}\left(h_{j_{\max}}^2+\sqrt{\mu_{j_{\max}}^{\mathrm{in}}}\right)\  .
\ee
\end{theorem}
\begin{proof}
With $\op{Z}^{\mathrm{in}}_{j}$ given by Algorithm~\ref{alg:multigrid}, 
denote by $e_j=d_{\sigma}(\op{Z}^{\mathrm{in}}_{j}, (\op{G}^{\mathrm{in}}_j)^{-1})$ and
by \mbox{$a_j=C \beta^{-1}(h_j^2+\sqrt{\mu_j^{\mathrm{in}}})$} with $C$ as in Theorem~\ref{maintheorem}. The assumption on
$\mu^{\mathrm{in}}_j$ implies that $a_j\le a_{j-1} \le f^{-1} a_{j}$. From Theorem~\ref{maintheorem} 
we know that $d_{\sigma}(\op{S}^{\mathrm{in}}_{j}, (\op{G}^{\mathrm{in}}_j)^{-1})\le a_j$.
Then, for $1\le j\le j_{\max}-1$
\bes
e_j&=& d_{\sigma}(\op{Z}^{\mathrm{in}}_{j}, (\op{G}^{\mathrm{in}}_j)^{-1}) =
d_{\sigma}(\mathfrak{N}_j(\mathfrak{I}_{j-1}^j (\op{Z}^{\mathrm{in}}_{j-1})), (\op{G}^{\mathrm{in}}_j)^{-1})\\
&\stackrel{\eqref{eq:sqconvnewt}}{\le}& C_1\left(d_{\sigma}(\mathfrak{I}_{j-1}^j (\op{Z}^{\mathrm{in}}_{j-1}),(\op{G}^{\mathrm{in}}_j)^{-1})\right)^2\\
&\le& C_1\left(d_{\sigma}(\mathfrak{I}_{j-1}^j (\op{Z}^{\mathrm{in}}_{j-1}),\op{S}^{\mathrm{in}}_j)+
d_{\sigma}(\op{S}^{\mathrm{in}}_j,(\op{G}^{\mathrm{in}}_j)^{-1}))\right)^2\\
&\stackrel{\eqref{eq:sdefop}}{=} &C_1\left(d_{\sigma}(\mathfrak{I}_{j-1}^j (\op{Z}^{\mathrm{in}}_{j-1}),\mathfrak{I}_{j-1}^j ((\op{G}^{\mathrm{in}}_{j-1})^{-1})+
d_{\sigma}(\op{S}^{\mathrm{in}}_j,(\op{G}^{\mathrm{in}}_j)^{-1}))\right)^2\\
&\stackrel{\eqref{eq:nonexp}}{\le} &C_1\left(d_{\sigma}(\op{Z}^{\mathrm{in}}_{j-1},(\op{G}^{\mathrm{in}}_{j-1})^{-1})+
a_j\right)^2 = C_1\left(e_{j-1}+a_j \right)^2\ .
\ees
By Lemma~\ref{eq:lma53} we have
\be
\label{eq:onetolast}
e_{j_{\max}-1}\le 4 C_1 a_{j_{\max}-1}^2\ .
\ee
Finally
\bes
e_{j_{\max}} &= &d_{\sigma}(\op{Z}^{\mathrm{in}}_{j_{\max}}, (\op{G}^{\mathrm{in}}_{j_{\max}})^{-1}) \le
d_{\sigma}(\op{Z}^{\mathrm{in}}_{j_{\max}}, \op{S}^{\mathrm{in}}_{j_{\max}})
+d_{\sigma}(\op{S}^{\mathrm{in}}_{j_{\max}}, (\op{G}^{\mathrm{in}}_{j_{\max}})^{-1})\\
&\stackrel{\eqref{eq:nonexp}}{\le} &
e_{j_{\max}-1}+a_{j_{\max}}\stackrel{\eqref{eq:onetolast}}{\le} 4 C_1 a_{j_{\max}-1}^2 + a_{j_{\max}}
\le (4 C_1 f^{-2}+1)a_{j_{\max}}\ ,
\ees
which proves~\eqref{eq:multigridest} with $C_2=(4 C_1 f^{-2}+1)$.
\end{proof}
\begin{remark}
\label{rem:mgthm}
As stated in Remark~$\ref{rem:sqrtest}$,  under certain,
natural conditions it is expected
that  $\mu_{j}^{\mathrm{in}}\approx C h_j$, case in which
the factor $f$ in Theorem~$\ref{thm:mganalysis}$  is $1/\sqrt{2}$.
However, while the estimate~$\eqref{eq:multigridest}$ is (asymptotically)
as good as the two-grid estimate in terms of $h_{j_{\max}}$ and 
$\mu^{\mathrm{in}}_{j_{\max}}$, the dependence of $C_2$ on $\beta$
is suboptimal. However, we believe that this is only an artifact of the analysis
of which the weak link is Lemma~$\ref{lma:sqnonsymest}$.
\end{remark}

\section{Numerical experiments}
\label{sec:numerics}
We test our algorithm on the standard linear elliptic-constrained distributed optimal control problem
\begin{eqnarray}
\label{mgipm:eq:lap}
  \begin{array}{cc}\vspace{7pt}
    \textnormal{minimize}& \frac{1}{2}\norm{y-y_d}^2 +
    \frac{\beta}{2}\norm{u}^2\\   \textnormal{subject to\ \ } &-\Delta
    y=u\ ,\ y\in H_0^1(\Omega),\  a\le u \le b\
    \  a.e.\  \mathrm{in}\ \Omega,
  \end{array}
\end{eqnarray}
where $\Delta$ is the Laplace operator acting on $H_0^1(\Omega)$, so $\op{K}=(-\Delta)^{-1}$.
The problem~\eqref{mgipm:eq:lap} is often used as a test example for multigrid 
algorithms in PDE constrained optimization (e.g., see~\cite{MR2505585, Dra:Pet:ipm}). 
Parts {\bf [a, b]} of 
Condition~\ref{mgipm:cond:condsmooth} follow from standard estimates for
finite element solutions of elliptic equations, while condition~{\bf [c]} follows 
from standard \mbox{$L^{\infty}$-estimates~\cite{MR2322235}.}
Note that the $L^{\infty}$-stability in {\bf [c]} does not necessitate an optimal 
\mbox{$L^{\infty}$-convergence} rate of the discrete solutions to their
continuous counterparts, but merely convergence:
$$\lim_{h_j\to 0} \nnorm{\op{K}_j u-\op{K} u}_{L^{\infty}} = 0\ .$$

We conduct three kinds of experiments to test different aspects of the method: first we consider 
an `in vitro' experiment where we artificially fix a nontrivial subset of the domain $\Omega$
based on which we construct  (artificial) inactive sets on all grids. The goal of this experiment is
to assess the change in quality of the two-grid preconditioner while varying only the mesh-size and the 
regularizing parameter~$\beta$. 
For each grid we then construct the two-grid preconditioner and we estimate numerically the
decay rate of the spectral distance $d_\sigma((\op{G}_j^{\mathrm{in}})^{-1}, \op{S}_j^{\mathrm{in}})$.
The advantage of this approach is that we can isolate the effect of the inactive set
on the quality of the preconditioner, since in the context of actually applying 
the SSNM, the inactive sets are expected to change from one resolution
to another. The second test is an actual, `in vivo' application of the two-preconditioner 
while solving~\eqref{mgipm:eq:lap} using the SSNM and we compare its performance with Hackbusch's multigrid method of the second kind.
In the third set of tests we compare actual runtimes of the linear solves  based on symmetric multigrid-preconditioned CG
with those performed with unpreconditioned CG in an actual SSNM solution process.


\subsection{\bf One-dimensional `in vitro' experiments}
\label{ssec:onednumexample}
Let $\Omega=[0,1]$ and designate the interval \mbox{$\Omega^{\mathrm{in}}=[1/8,3/4]$} as the set where `inactive vertices'
reside; this set is used for all grids. Note that in this experiment we do not solve a control problem, and the designated
inactive nodes do not correspond to a SSNM iterate.
Let $n_j=16\cdot 2^{j}, j=0, 1, 2, \dots$, and define $\op{T}_j$ to be the uniform grid
with $n_j$ intervals on $\Omega$ with mesh-size \mbox{$h_j=1/n_j$}. The `inactive indices'
are given by 
$$\op{I}_j^{\mathrm{in}} = \{i\in\{1,\dots,n_{j}-1\ :\ i\;h_j\in \Omega^{\mathrm{in}}\}\ .$$ 
Correspondingly, we construct the matrices 
$\mat{G}_j^{\mathrm{in}}$ and $\mat{M}_j^{\mathrm{in}}$ as in Section~\ref{sssec:matrixform} and we compute the 
quantities
$$d_j=\max\{ \abs{\ln \lambda}\ :\ \lambda\in \sigma(\mat{G}_j^{\mathrm{in}}, \mat{M}_j^{\mathrm{in}})\}\ ,\ j=1,2,\dots\ ,$$
where $\sigma(\mat{A},\mat{B})$ is the set of generalized eigenvalues of the matrices $\mat{A},\mat{B}$.
In general we have
\bes
d_j\le d_{\sigma}(\mat{G}_j^{\mathrm{in}}, \mat{M}_j^{\mathrm{in}})\ ,
\ees
but if both matrices are symmetric then the above inequality becomes an equality. In this case
$\mat{G}_j^{\mathrm{in}}$ is symmetric and $\mat{M}_j^{\mathrm{in}}$ is close to being symmetric, so we
expect that $d_j$ is a good approximation of $d_{\sigma}(\mat{G}_j^{\mathrm{in}}, \mat{M}_j^{\mathrm{in}})$.
For each of $\beta=1, 0.1, 0.01$ we report in Table~\ref{tab:specdist1d} both the numbers $d_j$ and their ratios $d_{j-1}/d_j$ for $j=2,\dots, 6$,
which clearly indicate that the numerical results are consistent
with the estimate $d_j \le C\: \sqrt{h_j}$ in Remark~\ref{rem:sqrtest}, namely we 
notice that $d_{j-1}/d_j\to \sqrt{2}$ for each $\beta$ considered.
\begin{table}
  \caption{Spectral distance decay for a fixed inactive domain $[1/8,3/4]\subset \Omega=[0,1]$.}
      {\begin{tabular}{@{}lcccccc}\hline
	  grid-size ($n_j$)&
	  $16$ & $32$ & $64$ & $128$ & $256$ & $512$\\\hline
	  $d_j$ ($\beta=1$) & $2.28\cdot 10^{-3}$ &  $1.56\cdot 10^{-3}$ & $1.08\cdot 10^{-3}$ & $7.46\cdot 10^{-4}$ & 
	  $5.20\cdot 10^{-4}$ & $3.64\cdot 10^{-4}$\\ \hline
	  $\frac{d_{j-1}}{d_j}$ ($\beta=1$) & -- &$1.4610$ & $1.4506$ &  $1.4421$ & $1.4351$ & $1.4295$\\
	  \hline
	  $d_j$ ($\beta=0.1$) & $2.18\cdot 10^{-2}$ &  $1.49\cdot 10^{-2}$ & $1.03\cdot 10^{-2}$ & $7.16\cdot 10^{-3}$ & 
	  $4.99\cdot 10^{-3}$ & $3.49\cdot 10^{-3}$\\ \hline
	  $\frac{d_{j-1}}{d_j}$ ($\beta=0.1$) & -- &$1.4589$ & $1.4491$ &  $1.4410$ & $1.4343$ & $1.4290$\\
	  \hline
	  $d_j$ ($\beta=0.01$) & $1.59\cdot 10^{-1}$ &  $1.10\cdot 10^{-1}$ & $7.66\cdot 10^{-2}$ & $5.34\cdot 10^{-2}$ & 
	  $3.74\cdot 10^{-2}$ & $2.62\cdot 10^{-2}$\\ \hline
	  $\frac{d_{j-1}}{d_j}$ ($\beta=0.01$) & -- &$1.4459$ & $1.4398$ &  $1.4343$ & $1.4296$ & $1.4257$\\
	  \hline
      \end{tabular}}
      \label{tab:specdist1d}
\end{table}

Another advantage of the one-dimensional example is that we can easily compute numerically 
all the generalized eigenvalues of $\mat{G}_j^{\mathrm{in}}$ and $\mat{M}_j^{\mathrm{in}}$. What may be surprising, is that, while
the largest of them decay at the predicted rate of $\sqrt{h_j}$ as $h_j\downarrow 0$, for each grid
most of the generalized eigenvalues are very close to $1$. In fact, for a fixed grid, if we order the 
generalized eigenvalues according to their distance from $1$, we notice an 
exponential decay of $\abs{\lambda_i-1}$ (see  Figure~\ref{fig:jointspec}). 
\begin{figure}[!htb]
\begin{center}
        \includegraphics[width=6in]{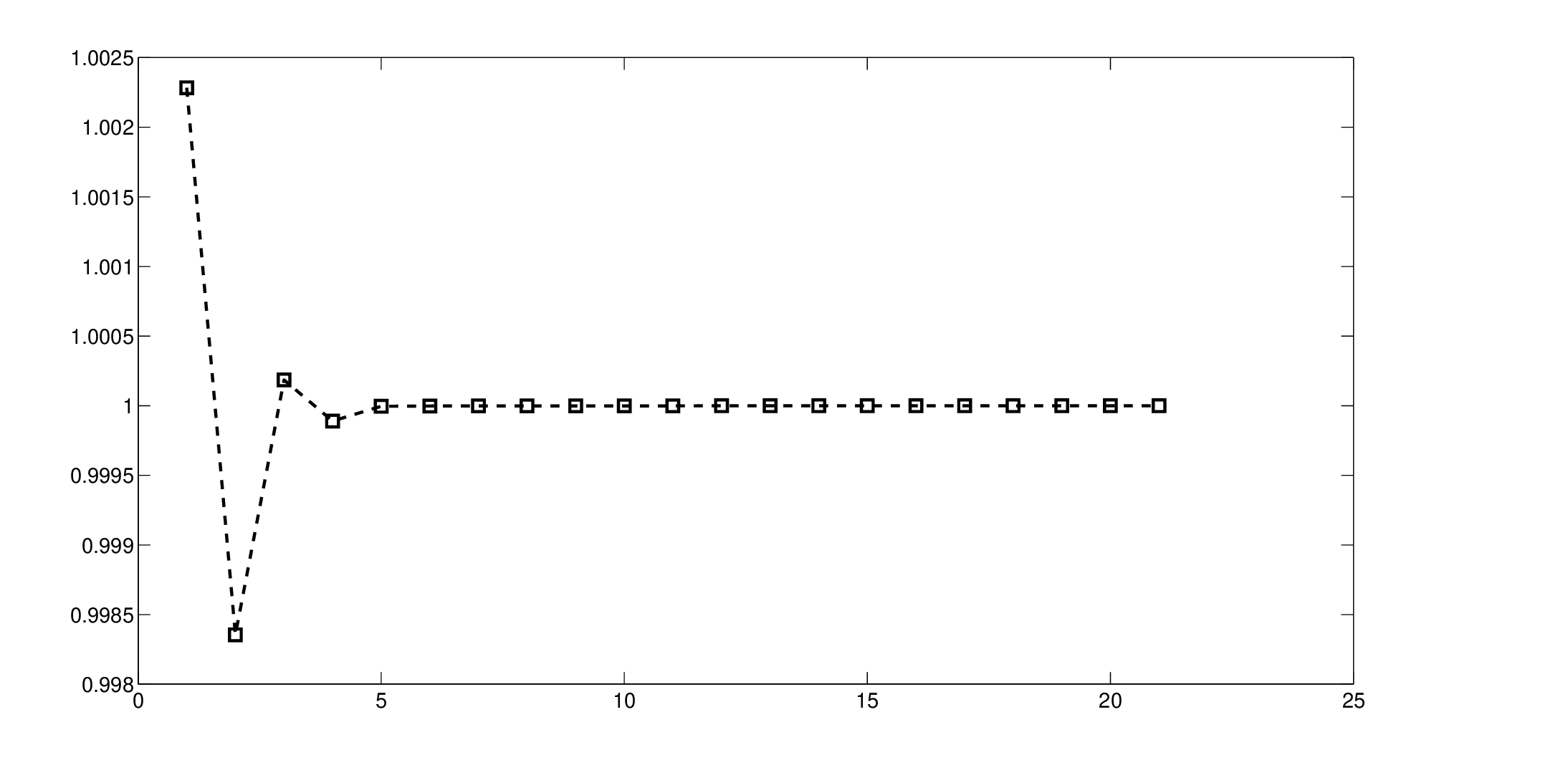}
\caption{Absolute values of the generalized eigenvalues of $\mat{G}_2^{\mathrm{in}}$ and $\mat{M}_2^{\mathrm{in}}$ 
$(n_2=32)$. There are $21$ eigenvalues, and their distance to $1$ decays rapidly.} 
\label{fig:jointspec}
\end{center}
\end{figure}
We also plot in Figure~\ref{fig:largesteig} the function $u\in \op{V}_2^{\mathrm{in}}$ associated with the generalized eigenvector~$\mat{u}$ 
that corresponds to the generalized eigenvalue which is furthest from $1$. What we notice is that $u$ is large around the 
boundary of $\Omega^{\mathrm{in}}$. These facts suggest that, with the exception of a limited number of generalized
eigenvalues associated with eigenvectors strongly related to the boundary of $\Omega^{\mathrm{in}}$, the generalized eigenvalues
$\mat{G}_j^{\mathrm{in}}$ and $\mat{M}_j^{\mathrm{in}}$ are in fact quite close 1. We revisit this idea in Section~\ref{ssec:remarks}.

\begin{figure}[!htb]
\begin{center}
        \includegraphics[width=6in]{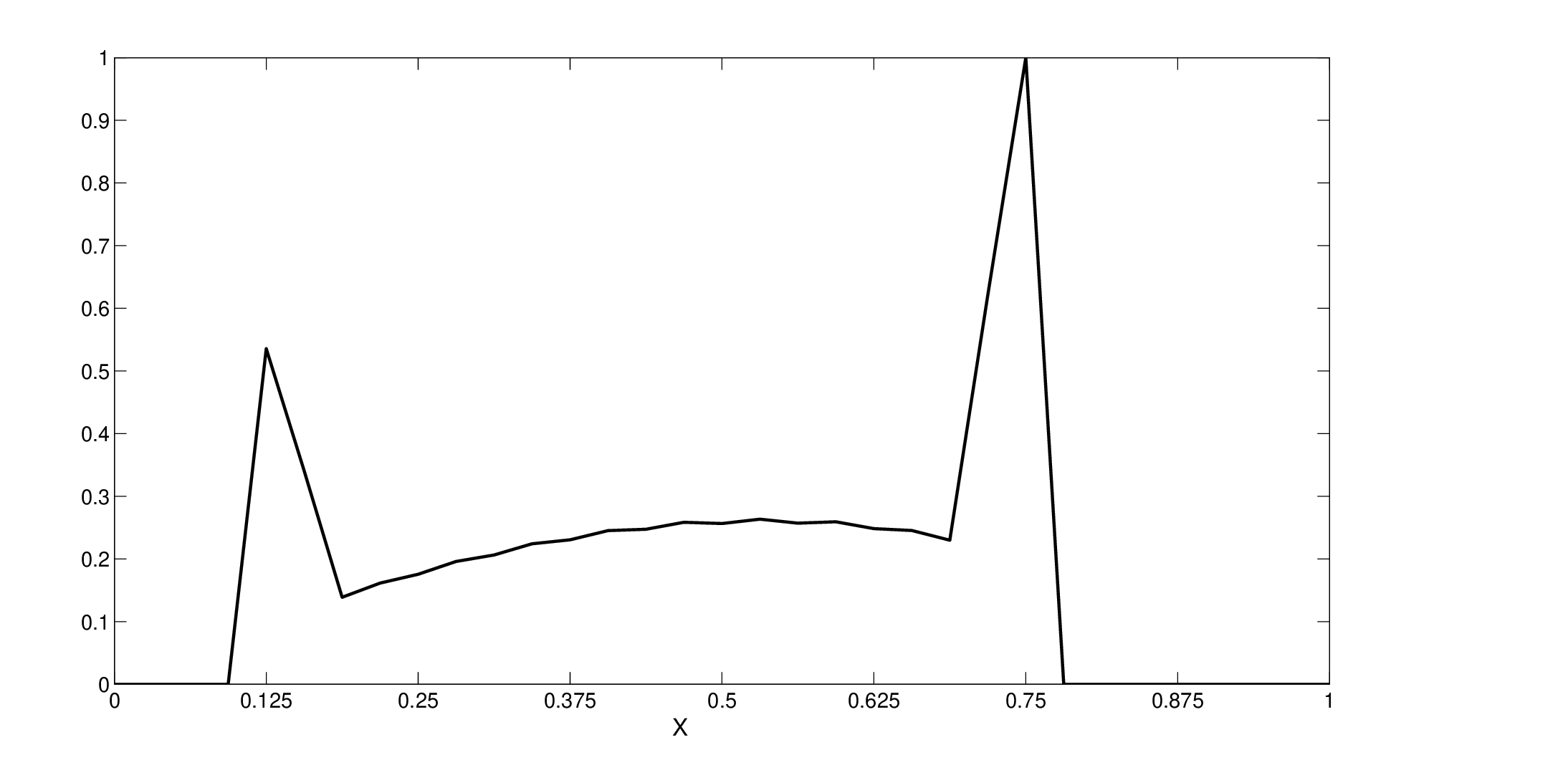}
\caption{The generalized eigenvector corresponding to the generalized eigenvalue $\mat{G}_2^{\mathrm{in}}$ and $\mat{M}_2^{\mathrm{in}}$
which is furthest from $1$.} 
\label{fig:largesteig}
\end{center}
\end{figure}

\subsection{Two-dimensional `in-vivo' experiments.}
\label{ssec:twodimnum}
We now consider the two-dimensional version of~\eqref{mgipm:eq:lap} with $\Omega=[0,1]^2$, and, for simplicity, we impose only the lower bound-constraint 
$a\equiv 0$. To construct 
uniform triangular grids $\op{T}_j$, $j=0,1,\dots,$ we first divide each side of $\Omega$ uniformly  in 
$n_j=64\cdot 2^{j}$ intervals to obtain a uniform rectangular grid, and we further divide each resulting 
grid-square in two triangles along the diagonal of slope $-1$; the mesh-size is thus $h_j=1/n_j$, and the number
of variables is $N_j=(n_j-1)^2$. The Poisson problem is then discretized using standard continuous piecewise linear 
elements on a triangular grid. We solve the problem using grid-sequencing, that is, the solution on level $(j-1)$ 
is used as an initial guess for the semismooth Newton iteration at level~$j$. More precisely, if $\Omega_{j-1}^{\mathrm{in}}$
denotes the coarse inactive domain as defined in Section~\ref{ssec:multigriddef}, we define the 
initial guess at the inactive  set on level $j$ by 
$$\op{I}_0^{(j)} = \{i\in\{1,\dots,N_j\}\ :\ \mathrm{supp}(\varphi_i^{(j)})\subseteq \Omega_{j-1}^{\mathrm{in}}\}\ .$$
Then we apply the semismooth Newton iteration described in Section~\ref{ssec:ssnm}, and we solve the reduced linear 
systems~\eqref{eq:disclinsysfin1}
at each outer iteration using CGS preconditioned by the two-grid
explicit preconditioner $\op{M}_j^{\mathrm{in}}$. For comparison, we also solve~\eqref{eq:disclinsysfin1}
(using the same coarse spaces and transfer operators defined in Section~\ref{ssec:multigriddef})
using Hackbusch's multigrid method of the second kind (e.g., see~\cite{MR814495}, Ch.~16, 
or~\cite{hackbusch81} for the original source). In order to do so we remark that the matrix 
in~\eqref{eq:disclinsysfin1} has the form $\mat{G}^{\mathrm{in},k} = \beta\mat{I}+ \mat{H},$
therefore we can rewrite the system $(\beta\mat{I}+ \mat{H})\mat{u}=\mat{r}$ as a fixed point problem:
\be
\label{eq:fixedpoint}
\vect{u}=\beta^{-1}\mat{r}-\beta^{-1}\mat{H} \mat{u}\ .
\ee
We also solve the systems~\eqref{eq:disclinsysfin1} using unpreconditioned 
CG. Since the solution process is matrix-free (we solve the Poisson problem using classical multigrid) 
we use the explicit form $\op{S}_j^{\mathrm{in}}$ of the inverse of $\op{M}_j^{\mathrm{in}}$ as given in~\eqref{eq:tgprecdefinv}. For each outer iteration
we record the number of two-grid preconditioned CGS iterations required for solving each linear system~\eqref{eq:disclinsysfin1}, as well
as the number of iterations in which the multigrid (here, two-grid) method of the second kind converged.
 Note that each CGS iteration requires two fine-grid matrix-vector multiplications
(mat-vecs), and the same holds for the multigrid  method of the second kind, so this should be a fair comparison.

For the numerical example we choose as `target' control the function defined by
$$u_d(x) = \left\{\begin{array}{cll}\vspace{5pt}r^{-4}(r^2-\nnorm{x-x_0}^2)+ \alpha&,&\ \mathrm{if}\ \nnorm{x-x_0}<r\ ,\\
\alpha&,&\  \mathrm{otherwise}\ \end{array}\right .$$
for some $x_0\in\Omega$ (see Figure~\ref{fig:ud}), and let $\widetilde{y}$ be the solution of the Poisson equation 
$$-\Delta \widetilde{y} = u_d,\ \ \widetilde{y}|_{\partial\Omega}\equiv 0\ .$$
The `data' $y_d$ entering the control problem~\eqref{mgipm:eq:lap} are obtained by adding to $\widetilde{y}$ a random perturbation that is uniformly distributed
in $[-\delta,\delta]$ with $\delta=0.05 \:\nnorm{\widetilde{y}}_{L^{\infty}}$.
If $u_d\ge 0$, and if $\delta$ and $\beta$  are small then the solution $u^{\min}$ of~\eqref{mgipm:eq:lap} is expected to be close to $u_d$. 
Therefore, by letting the parameter $\alpha$  be slightly negative we expect to find a localized $u^{\min}$ with nontrivial active and inactive sets, as desired
for testing the performance of our algorithm.
In Figure~\ref{fig:umin5} we show the solution~$u^{\min}$ for $\alpha=-0.1$ and $n_0=64$; as it turns out, for this example 
about $11.5$\% of the constraints are inactive if $\beta=10^{-5}$, and about $51\%$ are inactive if $\beta=10^{-4}$.

\begin{figure}[!htb]
\begin{center}
        \includegraphics[width=6in]{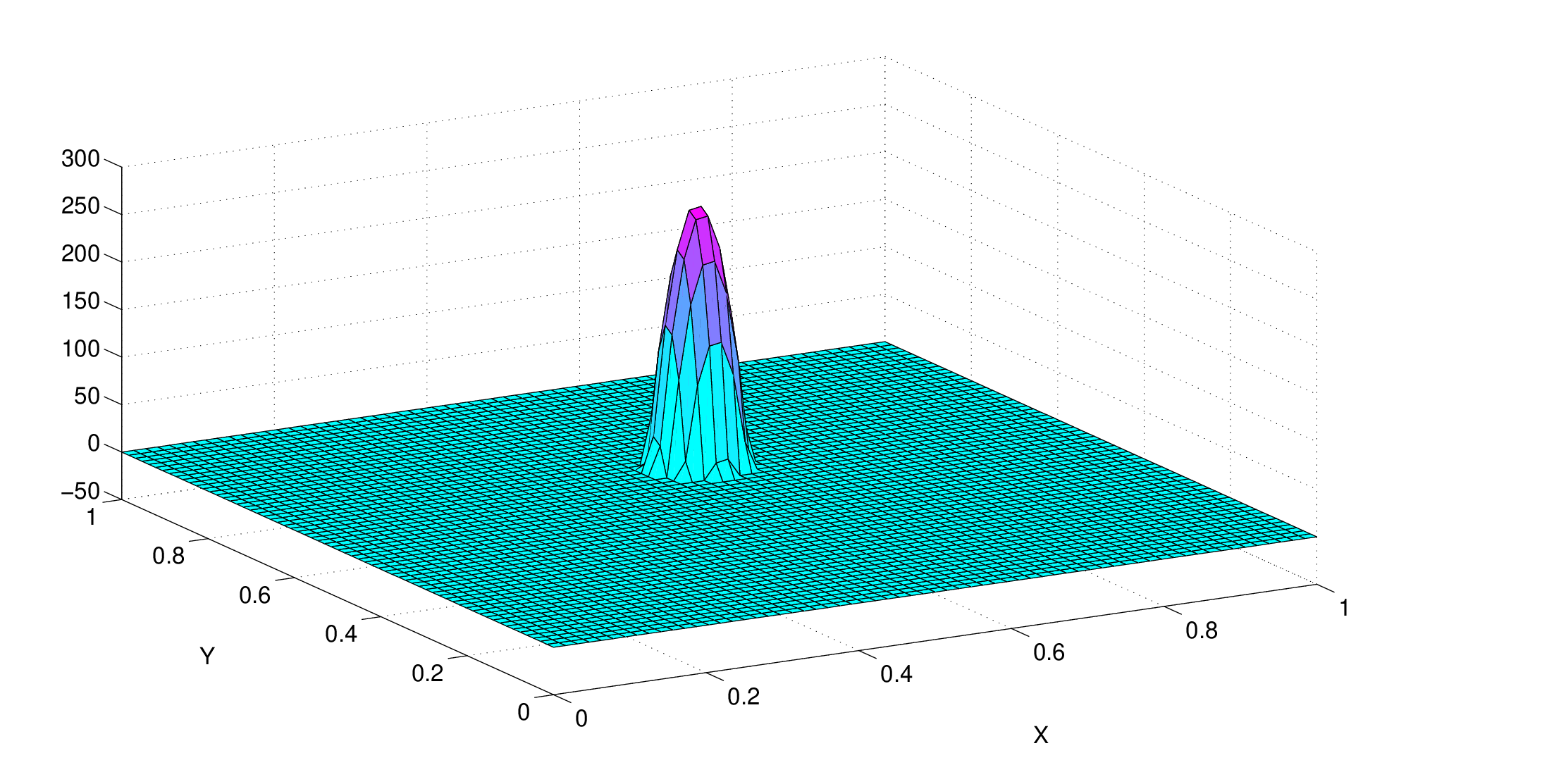}
\caption{Target control $u_d$ with $x_0=[0.54,0.62]^T$, $r=0.06$, $\alpha=-0.1$.} 
\label{fig:ud}
\end{center}
\end{figure}

\begin{figure}[!htb]
\begin{center}
        \includegraphics[width=6in]{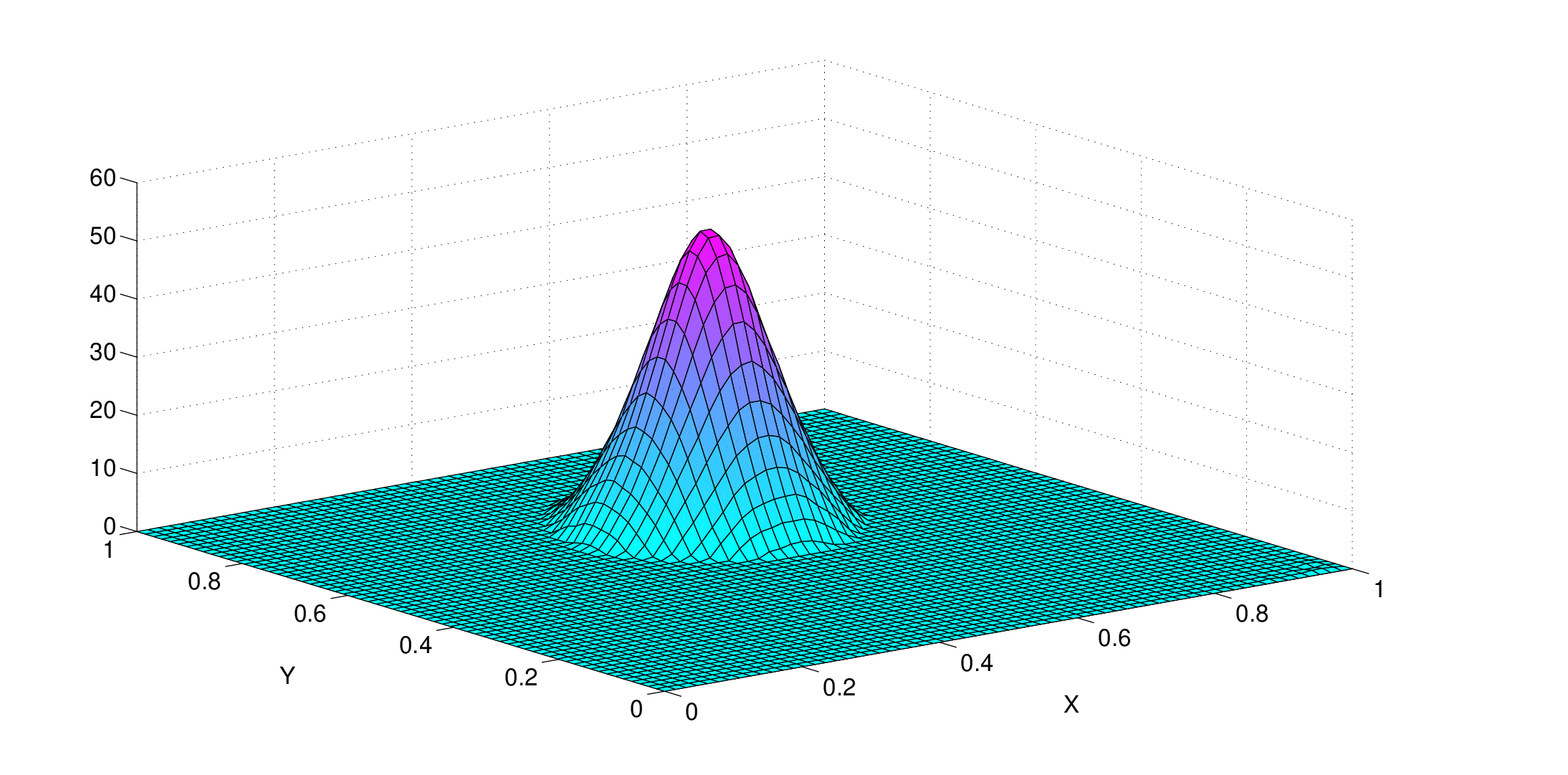}
\caption{Solution $u^{\min}$ on coarsest grid $(n_0=64)$ and $\beta=10^{-5}$; about $11.5\%$ of constraints are inactive.} 
\label{fig:umin5}
\end{center}
\end{figure}

The data in Table~\ref{tab:itcountssnm4} (respectively~\ref{tab:itcountssnm5})  
shows the iteration counts for the two-level preconditioned CGS method (columns labeled `A'), the 
multigrid method of the second kind (two levels only, columns `B'), and
for unpreconditioned CG (columns `C'), with the regularizing parameter being $\beta=10^{-4}$ (respectively $\beta=10^{-5}$). 
For example, in Table~\ref{tab:itcountssnm4} on the third grid ($n_2=256$) the SSNM converged in 4 iterations; for the first outer iteration the linear
system was solved in 4 two-grid preconditioned CGS iterations, while the remaining three outer iterations required 5 CGS iterations;
for the same systems the multigrid method of the second kind converged in 9 iterations, while CG required 13 iterations.
 First we should point out that the results shown  are consistent with the
theoretical estimates of Theorem~\ref{maintheorem} and Remark~\ref{rem:sqrtest}, namely the number of two-grid preconditioned CGS
iterations decreases with \mbox{$h_j\downarrow  0$}. There is one notable exception: for the very first iterate on the coarsest mesh we used
as initial guess $\op{I}_0^{(0)} = \{1,\dots,N_0\}$, so the two-grid preconditioner is the same as for the unconstrained problem, thus 
very efficient. In order to answer the question of whether the two-grid preconditioned CGS is more efficient than CG
we adopt the point of view that in a truly large-scale context (three- or four-dimensional problems, or when applying the operators $\op{K}_j$ requires 
a complicated sequence of operations) the cost of the solution process is largely given by fine-grid mat-vecs. If we accept
the number of fine-space mat-vecs as a measure 
of efficiency we notice that the two-grid preconditioned CGS solves required a total of 11 CGS iterations, hence 22 mat-vecs at the finest level
of~\mbox{$n_5=2048$}, compared to 36 mat-vecs needed by CG. As can be inferred from the tables, the ratio (\#~multigrid mat-vecs /  \# CG mat-vecs) is decreasing
with mesh-size. However, the above ratio would decrease much faster if  $d_\sigma((\op{G}_j^{\mathrm{in}})^{-1}, \op{S}_j^{\mathrm{in}})$
would decay at a faster rate, as it does in the unconstrained case. The multigrid method of the second kind also behaves as expected, in the sense
that the number of iterations decreases with $h_j\downarrow 0$. However, as the data suggest, it requires more iterations than the
two-grid preconditioned CGS method. 
As the theory suggests both for the preconditioner introduced in this work and for the multigrid method of the second kind,
the base level cannot be too coarse, or the algorithm will be inefficient or even fail to converge. As can be seen from Table~\ref{tab:itcountssnm5},
the multgrid method of the second kind only converges when $h_j\le 2^{-10}$, while the multigrid preconditioned CGS  method exhibits
variable behavior at coarser levels (but does converge). Finally, we should remark that the algorithm performed surprisingly well
considering that $\beta$ was in fact chosen quite small compared to what the theory suggested ($\beta\approx C \sqrt{h}$).

\begin{table}
  \caption{Iteration counts for the two-level preconditioned CGS method \textnormal{(A)}, Hackbusch  \textnormal{(B)}, and unpreconditioned CG  \textnormal{(C)}; $\beta=10^{-4}$.}
      {\begin{tabular}{|@{}c|ccc|ccc|ccc|ccc|ccc|ccc|}\hline
	   &
	  \multicolumn{3}{|c}{$64$} & \multicolumn{3}{|c}{$128$} & \multicolumn{3}{|c}{$256$} & \multicolumn{3}{|c}{$512$} & \multicolumn{3}{|c}{$1024$} 
	   & \multicolumn{3}{|c|}{$2048$}\\\hline
	  it. & A   & B  & C & A   &  B & C & A   & B & C & A   & B & C & A & B & C & A & B & C\\\hline
	  1& $2$ & 7  & 13 & $4$ & 11 & 12 & $4$ & 9 & 13 & $4$ & 7 & 13 & $3$ & 6 & 12 & $3$ & 6 & 12\\
	  2& $7$ & 16 & 13 & $6$ & 12 & 12 & $5$ & 9 & 13 & $4$ & 8 & 13 & $4$ & 7 & 12 & $4$ & 6 & 12\\
	  3& $6$ & 17 & 12 & $6$ & 12 & 12 & $5$ & 9 & 13 & $4$ & 8 & 13 & $4$ & 7 & 12 & $4$ & 6 & 12\\
	  4& $6$ & 17 & 12 & $6$ & 12 & 12 & $5$ & 9 & 13 & $4$ & 8 & 13 &  -- & -- & -- & --  & -- & --  \\
	  5& $6$ & 17 & 12 & --  & --   &  --  & --  & --  & --   & --  & -- &  --  &  --   & -- &  --  & --  & -- &  -- \\
	  \hline
      \end{tabular}}
      \label{tab:itcountssnm4}
\end{table}

\begin{table}
  \caption{Iteration counts for the two-level preconditioned CGS method  \textnormal{(A)}, Hackbusch  \textnormal{(B)}, and unpreconditioned CG  \textnormal{(C)}; $\beta=10^{-5}$; `nc' means `not converged'.}
      {\begin{tabular}{|@{}c|ccc|ccc|ccc|ccc|ccc|ccc|}\hline
	   &
	  \multicolumn{3}{|c}{$128$} &\multicolumn{3}{|c}{$256$} & \multicolumn{3}{|c}{$512$} & \multicolumn{3}{|c}{$1024$} 
	   & \multicolumn{3}{|c|}{$2048$}\\\hline
	  it. & A   &  B & C & A   & B & C & A   & B & C & A & B & C & A & B & C\\\hline
	  1& $4$ & nc  & 31 & $6$ & nc & 14 & $5$ & nc & 14 & $5$ & 18 & 12 & $4$ & 13 & 11\\
	  2& $11$& nc  & 16 & $30$& nc & 12 & $7$ & nc & 13 & $6$ & 21 & 11 & $5$ & 13 & 12\\
	  3& $11$& nc  & 15 & $9$ & nc & 12 & $7$ & nc & 12& $6$ & 21 & 11 & $5$ & 13 & 12\\
	  4& $9$ & nc  & 15 & $9$ & nc & 12 & $7$ & nc & 12& $6$ & 21 & 11 & $5$ & 13 & 12\\
	  5& $11$ & nc  & 17 & --  & -- & -- & --  & -- & --& -- & -- & -- & -- & -- & --\\
	  5& $16$ & nc  & 16 & --  & -- & -- & --  & -- & --& -- & -- & -- & -- & -- & --\\
	  5& $16$ & nc  & 16 & --  & -- & -- & --  & -- & --& -- & -- & -- & -- & -- & --\\
	  \hline
      \end{tabular}}
      \label{tab:itcountssnm5}
\end{table}

\subsection{Multigrid results}
\label{ssec:numericsmultigrid}
For our multigrid experiments with  problem~\eqref{mgipm:eq:lap} we used CG in conjunction with 
the multigrid preconditioner derived from the symmetric two-grid preconditioner~\eqref{eq:matformprecsym}. This time, in order
to show that our method is not restricted to triangular grids we discretize the 
PDE-constraints using continuous piecewise bilinear elements, and we eliminated any boundary restrictions on the controls;
 the coarse space is defined similarly 
to the one introduced in Section~\ref{ssec:multigriddef} (see~\eqref{eq:coarseinactdef}). 
Due to the problem sizes involved we solve Poisson's equation using standard multigrid -- the full approximation scheme -- with a base case of $256$ intervals in each direction; 
anytime we solve the PDE at a resolution lower than $256\times 256$ we employ a direct solver.
Recall that each Hessian-vector multiplication in the SSNM solution process 
requires two PDE solves. Note that the multigrid solve for Poisson's equation is completely
separated from the multigrid preconditioner for the optimization problem. The entire algorithm is implemented 
in M{\footnotesize ATLAB} (version R2010a) with the exception of the operation of identifying the coarse inactive nodes, which was implemented 
in C (and compiled with mex); the reason behind this choice is related to the unavoidable `for-loop' in this part of the algorithm which
adds significant runtime in an artificial manner to an otherwise simple routine. The computations were performed on a computer with two 8-core 2.9 GHz
Intel Xeon E5-2690 processors with 256 GB RAM.

In our experiments we measured the added  wall-clock time of all the linear solves in the SSNM solution process (including all the actions related
to setting up the multigrid preconditioner at each SSNM iteration), as well as the  wall-clock time for all the remaining operations. We should
point out that, aside from the linear solve, each SSNM  iterate requires two additional mat-vecs, regardless of the method chosen for solving the linear systems. 
Since our efforts were focused on
solving the linear systems, we compare the added runtimes for the multigrid preconditioned linear solves with those of the unpreconditioned CG
solves. 

Of all the experiments we conducted we show three that seem to represent the typical behavior of our method. In the first experiment we let $\beta=10^{-4}$, 
and we considered uniform $n\times n$ grids on $[0,1\times[0,1]$  with $n=64, 128, 512, 1024, 2048, 4096$. In order to ensure 
a sufficiently large active set for the solution we imposed a slightly elevated lower constraint, namely $a\equiv 0.1$. We set the tolerance for solving the
linear systems at $\tau=10^{-8}$. The results are reported in Table~\ref{tab:itcountssnm6}. In the row marked `\# cg its' we  list the number
of unpreconditioned CG iterations for each SSNM outer iteration. For example, for $n=1024$ there were three SSNM iterations, and at each SSNM
iteration the corresponding linear system required 11 unpreconditioned CG iterations to converge. Similarly, in the row marked 
`\# mg its ($n_0=64$)', {\bf the same} three linear systems required 8, 7, and 7 multigrid-preconditioned CG iterations; for all multigrid preconditioners
in this experiment the base case is set at $n_0=64$. Below the rows 
listing the number of iterations we show the added runtimes of the linear solves; for multigrid these include the time to set up the 
coarse spaces, and for constructing and inverting the coarsest Hessian. Finally in the row marked `$t_{\mathrm{mg}}/t_{\mathrm{cg}}$'  we record the
speed-up of the multigrid-preconditioned solves over the unpreconditioned solves. In terms of number of iterations we note 
an already familiar behavior: the number of unpreconditioned CG iterations remains bounded (for this example it turned out to be constant), while
the number of  multigrid-preconditioned CG solves decreases slowly with increasing resolution. The speed-up follows the same pattern as the 
number of multigrid preconditioned CG iterations, slowly decreasing from above-unity for small-to-medium resolutions (2.87 for two-grid preconditioners 
at $n_1=128$) to 0.67 at $n_8=4096$. Computing solutions at resolutions higher than $4096\times 4096$ with our current tools requires 
parallel programming and falls beyond the scope of this work.
\begin{table}[!h]
  \caption{Comparison of iteration counts and runtimes for multigrid vs. unpreconditioned CG; $\beta=10^{-4}$, $a=0.1, b=1$, $\mathrm{tol}=10^{-8}$.}
      {\begin{tabular}{|@{}c|c|c|c|c|c|c|}\hline
	  n & $128$ & $256$ & $512$ &  $1024$ & $2048$ & $4096$ \\\hline
	  \# cg its           & 11, 11, 11, 11 & 11, 11, 11, 11 & 11, 11, 11 & 11, 11, 11& 11, 11, 11 & 11, 11, 11\\
	  $t_{\mathrm{cg}}$ (s) & 4.24  & 25.75  & 104.26  & 490.74  &  2358   & 13553\\\hline
	  \# mg its ($n_0=64$) & 10, 9, 9, 9 & 10, 10, 10, 10  & 9, 9, 9& 8, 7, 7 & 8, 8, 8 & 7, 7, 7\\
	  $t_{\mathrm{mg}}$ (s) & 12.17  & 28.03  & 159.22  & 583.19  &  2454  & 9195  \\\hline
	  $t_{\mathrm{mg}}/t_{\mathrm{cg}}$&
	    2.87 & 1.09 & 1.53 & 1.19 & 1.04 & 0.68 \\
	  \hline
      \end{tabular}}
      \label{tab:itcountssnm6}
\end{table}

For the second set of computations (see Table~\ref{tab:itcountssnm7}) we decreased $\beta$ to $10^{-5}$ and chose the lower 
constraint $a\equiv 0$. We also lowered the tolerance $\tau=10^{-10}$ in order to increase the overall number of iterations.
While the behavior of unpreconditioned CG remains essentially consistent with the earlier example (the number of iterations
remains  fairly constant), the number of multigrid-preconditioned CG iterations with base $n_0=64$ shows a significant increase 
with increasing resolution.
This fact is consistent with the analysis, in that the estimates hold only if the base case is sufficiently fine. If we increase
the base case from 64 to 128 the `good' behavior is observed again, with the number of iterations slowly decreasing
with increasing resolution. Of course, the linear systems to be solved at the new base case $n_0=128$  are significantly larger than for base case $64$, 
and the runtimes are seriously affected at moderate resolutions, resulting in `speedups' of $\approx 41$ at $n=256$ and $\approx 8.4$ at $n=512$. However, 
at $n=4096$ we obtain an actual speedup of $0.63$. This trend suggests that at even higher resolutions multigrid-preconditioned CG will outperform
unpreconditioned CG even more.

\begin{table}[!h]
  \caption{Comparison of iteration counts and runtimes for multigrid vs. unpreconditioned CG; $\beta=10^{-5}$, $a=0, b=1$, $\mathrm{tol}=10^{-10}$.}
      {\begin{tabular}{|@{}c|c|c|c|c|c|}\hline
	   n & $256$ & $512$ &  $1024$ & $2048$ & $4096$\\\hline
	  \# cg its            & 23, 24, 23, 23 & 23, 24, 24  & 24, 24, 24 & 25, 24, 25   &   25, 25, 24     \\
	   $t_{\mathrm{cg}}$ (s) & 42.59  & 187.70   & 1034  &  5042  &   29385  \\\hline
	  \# mg its ($n_0=64$) & 20, 20, 19, 19 & 24, 24, 24 & 22, 23, 22 & 26, 26, 24 & 42, 49, 48\\\hline
	  \# mg its ($n_0=128$) & 18, 15, 15, 15 & 18, 17, 17  & 18, 18, 18 & 16, 14, 17 & 16, 16, 16 \\
	  $t_{\mathrm{mg}}$ (s) & 1747  & 1577  & 2700  & 5858  &  18580   \\\hline
	  $t_{\mathrm{mg}}/t_{\mathrm{cg}}$&
	    41.02 & 8.4 & 2.61 & 1.16 & 0.63 \\
	  \hline
      \end{tabular}}
      \label{tab:itcountssnm7}
\end{table}

\subsection{Further remarks}
\label{ssec:remarks}
We return to the numerical results from Section~\ref{ssec:onednumexample}, which 
suggest that the two-grid preconditioner is closer to being of optimal order than 
Theorem~\ref{maintheorem} predicts, in the following sense: if restricted to 
subspaces of functions supported away from the boundary of~$\Omega_j^{\mathrm{in}}$, namely
spaces of the form 
$$\op{V}_{j,
H}^{\mathrm{in}}=\mathrm{span}\{\varphi_i^{(j)}\in\op{V}_{j}^{\mathrm{in}}\
:\  \mathrm{dist}(\mathrm{supp}(\varphi_i^{(j)}),\partial
\Omega_j^{\mathrm{in}})\ge H\}\ ,$$ with $H>h$, then the
${\Hneg{2}(\Omega)}$-approximation property of
$\pi^{\mathrm{in}}_{j-1}$ is expected to be of almost optimal
order. To see this, consider $u\in \op{V}_{j, H}$  for
sufficiently large $H$. Then $\pi_{j-1}^{\mathrm{in}} u \approx 0$ on
$\partial_{\mathrm{n}}\Omega_j^{\mathrm{in}}$, since the size of
$\pi_{j-1}^{\mathrm{in}} u$ decays exponentially fast away from
$\mathrm{supp}(u)$, and therefore the numerical-boundary term
in~\eqref{eq:lemmapproxproj} from the proof of
Lemma~\ref{lma:l2projapprox} can be almost be dropped.
Since~\eqref{eq:l2projapprox} was the source of the largest term in
the estimate~\eqref{eq:twogridprecnormest} we expect that
$$ \nnorm{\mathrm {Proj}_{\op{V}_{j,
H}^{\mathrm{in}}}(\op{G}_j^{\mathrm{in}} -
\op{M}_j^{\mathrm{in}})|_{\op{V}_{j, H}^{\mathrm{in}}}} \le C h_j^2\ .
$$ 
While it is not difficult to formalize the above argument we should
point out its main consequence: the number of  generalized
eigenvalues in $\sigma(\op{G}_j^{\mathrm{in}},
\op{M}_j^{\mathrm{in}})$ that are significantly far from $1$ is
proportional to the number of nodes supported near the boundary of
$\Omega_j^{\mathrm{in}}$. If the dimension $d$ is greater than 1, then
this number normally increases with resolution, and certainly exceeds
the relatively low number of iterations we found in
Section~\ref{ssec:twodimnum}. This indicates that the number of
generalized eigenvalues which are $O(\sqrt{h_j})$ away from 1 dominate
the  computations.  While the two-level preconditioner does not have
optimal order, the above argument together with the analysis from
Section~\ref{sec:analysis} also suggest that in order to improve the
presented preconditioning technique one has to  tackle the diverging
action of the two operators on the nodal basis functions that are supported
 near $\partial\Omega_j^{\mathrm{in}}$.

\section{Conclusions}
\label{sec:conclusions}
We have constructed multigrid preconditioners for the linear systems
arising in the semismooth Newton solution process for a control
problem constrained by smoothing linear operators with box-constraints
on the control.  The preconditioner is similar
to the one constructed by Dr{\u a}g{\u a}nescu and Dupont
in~\cite{MR2429872} for the problem without inequality control-constraints,
and it maintains some  of its  qualitative behavior: its approximation properties improve
with increasing resolution. However, even though the approximation 
qualities of the constructed preconditioner are of suboptimal order,
the construction and the analysis form an important stepping stone
towards finding optimal order preconditioners for the systems under scrutiny.
The question of their practical importance and how they compare in
efficiency with the similar preconditioners developed by Dr{\u a}g{\u a}nescu and Petra 
in~\cite{Dra:Pet:ipm} for systems arising in interior point methods 
requires a more involved study and forms the subject of current research.

\section*{Acknowledgements}
The author is indebted to Todd Dupont for introducing him to this
subject and for guiding his initial steps in multigrid research and to
Michael Hinterm{\" u}ller for helpful discussions on semismooth Newton
methods. The author also thanks the referees for their thorough
reading of the manuscript and useful suggestions.  

\appendix
\section{Some technical results regarding the spectral distance}
\label{sec:specdist}
Given a Hilbert space~(${V},\innprd{\cdot}{\cdot})$ and an operator ${A}\in \mathfrak{L}_+({V})$ we define the bilinear form
$$\innprd{u}{v}_{{A}} = \innprd{{A} u}{v}\ .$$
For $B\in \mathfrak{L}({V})$ denote
$$w_{{A}}({B}) = \sup_{u\in{V}\setminus \{0\}}\Abs{\frac{\innprd{Bu}{u}_{A}}{\innprd{u}{u}_{A}}}$$
If $A$ is Hermitian, then $\innprd{\cdot}{\cdot}_{{A}}$ is  an inner product, case in which
$w_{{A}}({B})$ is the numerical radius of $B$ with respect to $\innprd{\cdot}{\cdot}_{{A}}$, and
the power inequality 
$$
w_A(B^n)\le (w_A(B))^n,\ \ n=1,2,3,\dots
$$
holds (e.g., see Theorem~2.1-1 in~\cite{MR98b:47008}).
The next result is a weak generalization of the inequality above 
for the case when $A$ is non-symmetric and $n=2$.
\begin{lemma}
\label{lma:sqnonsymest}
There exists a constant $C_{A}=C(\nnorm{A_a^{-1}}, \kappa(A), \nnorm{A})$ so that
\be
\label{eq:pow2ineq}
w_A(B^2)\le C_A (w_A(B))^2,
\ee
where $A_s, A_a$ are the Hermitian and skew-Hermitian parts of $A$, and \mbox{$\kappa(A)=\nnorm{A}\cdot \nnorm{A^{-1}}$}.
\end{lemma}
\begin{proof}
First we identify two constants $c_1, c_2$ so that $c_1\nnorm{B}\le w_A(B)\le c_2 \nnorm{B}$.
Let $\alpha=\nnorm{A_a^{-1}}^{-1}>0$ be the smallest eigenvalue of $A_a$. 
We have 
\be
\Abs{\frac{\innprd{B u}{u}_A}{\innprd{ u}{u}_A}} = \Abs{\frac{\innprd{A B u}{u}}{\innprd{A u}{u}}}\le \alpha^{-1}\nnorm{A}\cdot\nnorm{B}\ ,
\ee
since $\Abs{\innprd{A u}{u}}=\Abs{\innprd{A_s u}{u} + \innprd{A_a u}{u}}\ge \alpha \nnorm{u}^2$, because  $\innprd{A_a u}{u}\in \R$ 
and $\innprd{A_a u}{u}\in {\bf i}\R$.
By taking the supremum over all $u\in V\setminus\{0\}$ we obtain
\be
w_A(B)\le \alpha^{-1}\nnorm{A}\cdot\nnorm{B}\ .
\ee
From the polarization identity
\bes
 4\innprd{B u}{v}_A &= &\innprd{B(u+v)}{u+v}_A - \innprd{B(u-v)}{u-v}_A \\
&&+{\bf i}\innprd{B(u+{\bf i}v)}{u+{\bf i}v}_A -{\bf i}\innprd{B(u-{\bf i}v)}{u-{\bf i}v}_A 
\ees
we obtain
\bes
4\Abs{\innprd{B u}{v}_A}&\le & w_{A}(B)\nnorm{A}\left(\nnorm{u+v}^2+\nnorm{u-v}^2+\nnorm{u+{\bf i}v}^2 +\nnorm{u-{\bf i}v}^2 \right)\\
&=& 4 w_{A}(B)\nnorm{A}\left(\nnorm{u}^2+\nnorm{v}^2\right)\ .
\ees
By taking the supremum over all unit vectors $u, v\in V$ we obtain
$\nnorm{A B}\le  2 w_{A}(B)\nnorm{A}$. Hence
$$
\nnorm{B}\le \nnorm{AB}\nnorm{A^{-1}}\le 2 w_{A}(B)\kappa(A)\ .
$$
So $c_1=(2\kappa(A))^{-1}$ and $c_2=\alpha^{-1}\nnorm{A}$.
Therefore
\bes
w_A(B^2)\le c_2 \nnorm{B^2}\le c_2 \nnorm{B}^2\le c_2 c_1^{-2}w_A(B)^2\ ,
\ees
so $C_A =  c_2 c_1^{-2}$.
\end{proof}

We should point out that if $A$ is Hermitian, then, cf.~\cite{MR98b:47008} we have $C_A=1$.
Mounting numerical evidence suggests that $C_A$ is bounded uniformly with respect to
$A$ even for the non-Hermitian case; however,
proving this claim is still part of ongoing research.

We now prove Lemma~\ref{lma:techlma1}.
\begin{proof}
For two complex numbers $z_1, z_2\in \H$ we have
\be
\label{eq:complezineq}
\Re\left(\overline{(z_2-z_1)}(\ln z_2 - \ln z_1)\right)>0\ .
\ee
To verify~\eqref{eq:complezineq} consider the segment
$\gamma:[0,1]\to \C$ given by $\gamma(t)=z_1+t (z_2-z_1)$. Then
\bes
\lefteqn{\Re\left(\overline{(z_2-z_1)}(\ln z_2 - \ln z_1)\right)}\\
&=& \Re\left(\overline{(z_2-z_1)}\int_{\gamma}\frac{dz}{z}\right)
= \Re\left(\abs{z_2-z_1}^2\int_0^1\frac{dt}{z_1+t (z_2-z_1)}\right)\ ,
\ees
and~\eqref{eq:complezineq} follows from $(z_1+t (z_2-z_1))^{-1}\in \H$ for $t\in [0,1]$.
Given arbitrary $w_1, w_2\in \H$ define the function
$$f:[0,\infty)\to \R, \ \ f(x)=\Abs{\ln\frac{w_1+x}{w_2+x}}^2\ .$$ 
We have
$$f'(x) = 2\: \Re\left((\overline{(w_1+x)^{-1}}-\overline{(w_2+x)^{-1}})(\ln (w_1+x)-\ln(w_2+x))\right)<0,$$
as follows from~\eqref{eq:complezineq} with $z_i=(w_i+x)^{-1}, i=1, 2$. Therefore $f$ is decreasing 
and~\eqref{eq:w12xineq} follows.
\end{proof}

We now conclude with the proof of Lemma~\ref{lma:sqconvnewt}.
\begin{proof}
First note that 
\bes
\mathfrak{N}_j(\op{X})-(\op{G}^{\mathrm{in}}_j)^{-1} = 2 \op{X} - \op{X} \op{G}^{\mathrm{in}}_j\op{X} - (\op{G}^{\mathrm{in}}_j)^{-1}
=-(\op{G}^{\mathrm{in}}_j)^{-1}(\op{G}^{\mathrm{in}}_j \op{X}-I)^2\ .
\ees
Since for $z$ around $1$ we have $\ln  z = z-1 + O(z-1)$ there are $k>1$ and $\delta>0$ so that, if $z\in \op{B}_\delta(1)$ or $\abs{\ln z}\le \delta$, then
$$k^{-1}|z-1|\le \abs{\ln z}\le k |z-1| .$$
Denote by $A=(\op{G}^{\mathrm{in}}_j)^{-1}$ and assume that $d_{\sigma}(\op{X},A)\le \delta$ with $\delta$ to be determined.
For $u\in \op{V}_{j}^{\C}\setminus \{0\}$ and $A=(\op{G}^{\mathrm{in}}_j)^{-1}$
\bes
\Abs{\frac{\innprd{\mathfrak{N}_j(\op{X}) u}{u}}{\innprd{A u}{u}} - 1} &=& 
\Abs{\frac{\innprd{A(\op{G}^{\mathrm{in}}_j \op{X}-I)^2 u}{u}}{\innprd{A u}{u}}}
\le w_{A}\left((\op{G}^{\mathrm{in}}_j \op{X}-I)^2\right)\\
&\le& C_A \left(w_A(\op{G}^{\mathrm{in}}_j \op{X}-I)\right)^2\ ,
\ees
where $C_A$ is the constant in from Lemma~\ref{lma:sqnonsymest} (depending on $\beta, C$).
Since 
\bes
w_A(\op{G}^{\mathrm{in}}_j \op{X}-I) = \sup_{u\in \op{V}_{j}^{\C}\setminus \{0\}} \Abs{\frac{\innprd{\op{X}u}{u}}{\innprd{A u}{u}}-1}\le
k\: d_{\sigma}(\op{X},A)\ ,
\ees
we obtain
\bes
\Abs{\frac{\innprd{\mathfrak{N}_j(\op{X}) u}{u}}{\innprd{A u}{u}} - 1} \le C_A k^2 \left(d_{\sigma}(\op{X},A)\right)^2 \le C_A k^2 \delta^2\ .
\ees
By further restricting $\delta$ so that $C_A k^2 \delta^2 <\delta$ we get
\bes
\sup_{u\in \op{V}_{j}^{\C}\setminus \{0\}}\Abs{\ln\frac{\innprd{\mathfrak{N}_j(\op{X}) u}{u}}{\innprd{A u}{u}}}&\le& 
k \sup_{u\in \op{V}_{j}^{\C}\setminus \{0\}}\Abs{\frac{\innprd{\mathfrak{N}_j(\op{X}) u}{u}}{\innprd{A u}{u}} - 1} \le
C_A k^3(d_{\sigma}(\op{X},A))^2\ .
\ees
\end{proof}

\bibliography{ssnmrefs}
\bibliographystyle{siam}

\end{document}